\documentclass[reqno,11pt]{article}
\usepackage{mathrsfs}
\usepackage{amssymb}
\usepackage{amsthm}
\usepackage{amsmath}
\usepackage{amsfonts}
\usepackage{mathtools}
\usepackage{enumerate}
\usepackage{bm}

\usepackage{graphicx}

\usepackage{bbm}
\usepackage{mathrsfs}
\usepackage{amssymb}
\usepackage[thicklines]{cancel}

\usepackage[usenames,dvipsnames]{xcolor}
\usepackage{soul}

\usepackage{txfonts}

\usepackage{psfrag}

\usepackage{setspace,color,wrapfig}
\usepackage[colorlinks,linkcolor=blue,citecolor=cyan]{hyperref}

\numberwithin{equation}{section}
\newtheorem{theorem}{Theorem}[section]
\newtheorem{lemma}[theorem]{Lemma}
\newtheorem{corollary}[theorem]{Corollary}
\newtheorem{proposition}[theorem]{Proposition}

\theoremstyle{definition}
\newtheorem{definition}[theorem]{Definition}
\newtheorem{assumption}[theorem]{Assumption}
\newtheorem{example}[theorem]{Example}

\theoremstyle{remark}
\newtheorem{remark}[theorem]{Remark}

\begin{document}

\def\be{\begin{eqnarray}}
\def\ee{\end{eqnarray}}
\def\p{\partial}
\def\no{\nonumber}
\def\eps {\epsilon}
\def\de{\delta}
\def\De{\Delta}
\def\om{\omega}
\def\Om{\Omega}
\def\f{\frac}
\def\th{\theta}
\def\la{\lambda}
\def\lab{\label}
\def\b{\bigg}
\def\var{\varphi}
\def\na{\nabla}
\def\ka{\kappa}
\def\al{\alpha}
\def\La{\Lambda}
\def\ga{\gamma}
\def\Ga{\Gamma}
\def\ti{\tilde}
\def\wti{\widetilde}
\def\wh{\widehat}
\def\ol{\overline}
\def\ul{\underline}
\def\Th{\Theta}
\def\si{\sigma}
\def\Si{\Sigma}
\def\oo{\infty}
\def\q{\quad}
\def\z{\zeta}
\def\M{\mathbf}
\def\co{\coloneqq}
\def\eqq{\eqqcolon}
\def\bt{\begin{theorem}}
\def\et{\end{theorem}}
\def\bc{\begin{corollary}}
\def\ec{\end{corollary}}
\def\bl{\begin{lemma}}
\def\el{\end{lemma}}
\def\bp{\begin{proposition}}
\def\ep{\end{proposition}}
\def\br{\begin{remark}}
\def\er{\end{remark}}
\def\bd{\begin{definition}}
\def\ed{\end{definition}}
\def\bpf{\begin{proof}}
\def\epf{\end{proof}}
\def\bex{\begin{example}}
\def\eex{\end{example}}
\def\bq{\begin{question}}
\def\eq{\end{question}}
\def\bas{\begin{assumption}}
\def\eas{\end{assumption}}
\def\ber{\begin{exercise}}
\def\eer{\end{exercise}}
\def\mb{\mathbb}
\def\mbR{\mb{R}}
\def\mbZ{\mb{Z}}
\def\mbf{\mathbf}
\def\u{\mbf{u}}
\def\v{\mbf{v}}
\def\U {\mbf{U}}
\def\e{\mbf{e}}
\def\x{\mbf{x}}
\def\mc{\mathcal}
\def\mcS{\mc{S}}
\def\lan{\langle}
\def\ran{\rangle}
\def\lb{\llbracket}
\def\rb{\rrbracket}
\def\div {\mathrm{div}\,}
\def\curl {\mathrm{curl}\,}

\title{ Structural stability of three dimensional transonic shock flows with an external force}
\author{Shangkun Weng\thanks{School of mathematics and statistics, Wuhan University, Wuhan, Hubei Province, 430072, People's Republic of China. Email: skweng@whu.edu.cn}\and
Zihao Zhang\thanks{School of mathematics and statistics, Wuhan University, Wuhan, Hubei Province, 430072, People's Republic of China. Email: zhangzihao@whu.edu.cn} \and Yan Zhou\thanks{School of mathematics and statistics, Wuhan University, Wuhan, Hubei Province, 430072, People's Republic of China. Email: yanz0013@whu.edu.cn}}
\date{}
\maketitle
 \newcommand{\ms}{\mathcal{S}}
 \newcommand{\mn}{\mathcal{N}}
\begin{abstract}
  We establish the existence and uniqueness of the transonic shock solution for steady isentropic Euler system with an external force in a rectangular cylinder under the three-dimensional perturbations for the incoming supersonic flow, the exit pressure and the external force. The external force has a stabilization effect on the transonic shocks in flat nozzles and the transonic shock is completely free, we do not require it passing through a fixed point. By utilizing the deformation-curl decomposition to decouple the hyperbolic and elliptic modes in the steady Euler system effectively and reformulating the Rankine-Hugoniot conditions, the transonic shock problem is reduced to a deformation-curl first order system for the velocity field with nonlocal terms supplementing with an unusual second order differential boundary condition on the shock front, an algebraic equation for determining the shock front and two transport equations for the Bernoulli's quantity and the first component of the vorticity.
\end{abstract}

\begin{center}
\begin{minipage}{5.5in}
Mathematics Subject Classifications 2020: 35L67, 35M12, 76H05, 76N15.\\
Key words:  transonic shocks,  stabilization effect on the external force,  the deformation-curl decomposition, Rankine-Hugoniot conditions.
\end{minipage}
\end{center}

\section{Introduction and  main results}\noindent
\par In this paper, we investigate the   transonic shock problem for steady Euler flows in a rectangular cylinder under the external force,
which  is governed by the following  system:
 \begin{equation}\label{1-1}
\begin{cases}
\p_{x_1}(\rho u_1)+\p_{x_2}(\rho u_2)+\p_{x_3}(\rho u_3)=0,\\
\p_{x_1}(\rho u_1^2+P)+\p_{x_2}(\rho u_1u_2)+\p_{x_3}(\rho u_1u_3)=\rho \p_{x_1}\Phi,\\
\p_{x_1}(\rho u_1u_2)+\p_{x_2}(\rho u_2^2+P)+\p_{x_3}(\rho u_2u_3)=\rho \p_{x_2}\Phi,\\
\p_{x_1}(\rho u_1u_3)+\p_{x_2}(\rho u_2u_3)+\p_{x_3}(\rho u_3^2+P)=\rho \p_{x_3}\Phi.\\
\end{cases}
\end{equation}
Here $ \textbf{u} = (u_1, u_2,u_3) $ is the velocity field, $ \rho $ is the density, $ P $ is the pressure and $\Phi $ is the potential force. We only consider the isentropic polytropic gases, therefore the equation of state is given by $ P=A\rho^{\gamma} $, where  $ A $ is a positive constant and $ \gamma> 1 $ is the adiabatic constant. For convenience, we take $A=1$ in this paper. Denote the sound speed by  $ c(\rho)=\sqrt{P^{\prime}(\rho)}$. The  system \eqref{1-1} is  hyperbolic  for supersonic flows (i.e. $ |\textbf{u}|>c(\rho) $) and  hyperbolic-elliptic coupled  for subsonic flows (i.e. $ |\textbf{u}|<c(\rho) $). 
\par 
There have been many results on transonic shock flows in nozzles for various situations. The existence of planar normal shock solutions can be easily established for steady flows in finitely and infinitely long flat nozzles. However, the position of the shock front can be arbitrary in the flat nozzle. How to determine the position of the shock front uniquely is a crucial issue. In \cite{CGF03}, the authors proved the existence of transonic shock in the finitely long nozzles for multi-dimensional potential flows with a given potential value at the exit by assuming that the shock front passed through a given point. Under the same restriction, the authors in \cite{ XY05,YX08} proved some ill-posedness results for the potential flows with the exit pressure. For the steady full Euler equations on 2-D nozzles with slowly varying cross-sections, \cite{CCS06,CCM07} proved the existence of a transonic shock. The stability of the transonic shock to steady Euler flows under 2-D perturbation was investigated in \cite{CS05,XYY09}. The existence and stability of transonic shocks in perturbed 3-D compressible flows passing a duct were studied in \cite{CS08, CY08}, and this conclusion is also established under the assumption that the shock front passes through a given point. The authors in \cite{FX21} eliminated the artificial assumption and established the existence of transonic shock solutions to the 2D compressible Euler system in almost flat nozzles.

 \par  On the other hand, some significant progress for the well-posedness of the transonic shock problem in two dimensional divergent nozzles under the perturbations of the exit pressure were established in \cite{LXY09,LYX09,XY08}. It is shown in \cite{LXY13} that the transonic shock in a 2-D straight divergent nozzle is structurally stable by perturbations of the nozzle wall and outlet pressure. In \cite{LYX10,LXY10}, the existence of a transonic shock for three-dimensional axisymmetric flows without swirl in a conical nozzle is demonstrated to be structurally stable under appropriate perturbations of the outlet pressure. For the structural stability under axisymmetric perturbation of the nozzle wall, a modified Lagrangian coordinate was introduced in \cite{WXX21} to deal with both corner singularities near the intersection points of the shock surface and the nozzle boundary and artificial singularity near the axis. Furthermore, the authors in \cite{WXY21,WYX21} studied the radial symmetric flows with nonzero angular velocity with or without shock in an annulus. The stability of spherically symmetric transonic shocks in a spherical shell was studied in \cite{LXY16} by requiring that the background shock solution satisfies the ``Structural Condition" which seems difficult to verify. Recently, Weng and Xin in \cite{WX23} had made a substantial progress and established the existence and stability of cylindrical transonic shock solutions under three dimensional perturbations of the incoming flows and the exit pressure without any restriction on the background transonic shock solutions. 

\par Let $ L_0 $, $ L_1(>L_0) $ be fixed positive constants. The nozzle is given by
\begin{equation*}
\mn= \{(x_1,x_2,x_3): L_0<x_1<L_1, \  (x_2,x_3)\in E\},\quad E=(-1,1)\times (-1,1).
\end{equation*}

First, we focus on the 1-D steady transonic shock flow with an external force. Namely, we solve
\begin{equation}\label{1-2}
\begin{cases}
		(\bar{\rho} \bar{u})'(x_1)=0,\\
		\bar{\rho} \bar{u} \bar{u}'+ \frac{d}{dx_1} P(\bar{\rho})= \bar{\rho}
		\bar{f}(x_1),\\
		\bar{\rho}(L_0)=\rho_0>0,\ \ \bar{u}(L_0)= u_0>0,\\
        \bar{P} (L_1) = P_e,
\end{cases}
\end{equation}
where  the flow state at the entrance $x_1=L_0$ is supersonic, i.e., $u_0^2>c^2(\rho_0)=\gamma
\rho_0^{\gamma-1}$. By employing the monotonicity relation between the shock position and the end pressure, the following lemma was established in \cite{WY22} to show that there is a unique transonic shock solution to \eqref{1-2} when the end pressure is a suitably prescribed constant $P_e$ and $f (x_1) > 0$ for any $ x_1 \in [L_0, L_1]$. Meanwhile, it is shown that the external force has a stabilization effect on the transonic shock in the nozzle and the shock position is uniquely determined.

\begin{lemma} Suppose that the initial state $(u_0,\rho_0)$
at $x_1=L_0$ is supersonic and the external force $ \bar{f}>0 $ for any $x_1 \in [L_0, L_1]$, there are two positive constants $P_0,P_1$ such that if the end pressure $P_e\in (P_1, P_0)$, there exists a unique piecewise transonic shock solution
\begin{equation}\no
\mathbf{\bar\Psi}(\mathbf{x})=(\mathbf{\bar u},\bar \rho)(\mathbf{x})=
\begin{cases}
\mathbf{\bar\Psi}^-(\mathbf{x})=(\bar u^-(x_1),0,0,\bar \rho^-(x_1))\quad {\rm{if}}\ L_0<x_1<L_s,\\
\mathbf{\bar\Psi}^+(\mathbf{x})=(\bar u^+(x_1),0,0,\bar \rho^+(x_1))\quad {\rm{if}}\ L_s<x_1<L_1,
\end{cases}
\end{equation}
 with a shock located at $x_1=L_{s}\in (L_0,L_1)$. Across the shock, the Rankine-Hugoniot conditions and entropy condition are satisfied:
 \begin{equation*}
 \begin{cases}
  [{ \bar\rho} {\bar u}](L_s)=[{\bar\rho} { \bar u}^2+\bar P](L_s)=0,\\
 [\bar P](L_s)>0,
 \end{cases}
 \end{equation*}
 where $[g] (L_s) \co g(L_s+0) - g(L_s-0)$ denotes the jump of $g$ at $x_1 = L_s$. In addition, the shock position $x_1=L_s$ increases as the exit pressure $P_{e}$
decreases. The shock position $L_s$ approaches to $L_1$ if $P_{e}$ goes to $P_1$ and $L_s$ tends to $L_0$ if $P_{e}$ goes to $P_0$.
\end{lemma}
\par In this paper, the 1-D transonic shock solution $ \mathbf{\bar\Psi}$ where the shock occurs at $ x_1=L_s $ is called the background solution. Clearly, one can extend the supersonic and subsonic parts of $\mathbf{\bar\Psi} $ respectively in a natural way. For convenience, we will still call the extended subsonic and supersonic solutions $ \mathbf{\bar\Psi}^+$ and $\mathbf{\bar\Psi}^- $. This paper is going to establish the structural stability of this transonic shock solution under three perturbations of the incoming supersonic flows, the exit pressure and the external force.
\par   Suppose that the potential force $ \Phi $ and the supersonic incoming flow at the inlet $ x=L_0 $ are
\begin{equation}\label{1-12}
\begin{cases}
\Phi(x_1,x_2,x_3)=\bar \Phi(x_1)+\epsilon \Phi_0(x_1,x_2,x_3),\\
\mathbf{\Psi}^-(L_0,x_2,x_3)=\mathbf{\bar\Psi}^-(L_0)
+\epsilon(u_{10}^-, u_{20}^-,u_{30}^-, P_{0}^-)(x_2,x_3).
\end{cases}
\end{equation}
Here  $ \bar\Phi^{\prime}=\bar f $ and
$ \epsilon $ is sufficiently small. Furthermore, $ \Phi_0 (x_1,x_2,x_3)\in C^{2,\alpha}(\overline{ \mn}) $ and $ (u_{10}^-, u_{20}^-, u_{30}^-, P_{0}^-)\in (C^{2,\alpha}(\overline{ E}))^4 $ satisfy the following compatibility conditions
\begin{equation}\label{bc13}
\begin{cases}
  \p_{2}\Phi_0(x_1, \pm 1, x_3)= u_{20}^- (\pm 1, x_3)  = \p_{2}^2 u_{20}^- (\pm 1, x_3) =  \p_{2} (u_{10}^-, u_{30}^-, P_0^-) (\pm 1, x_3) = 0, \q \forall x_3 \in [- 1, 1], \\
 \p_{3}\Phi_0(x_1,x_2, \pm 1)= u_{30}^- (x_2, \pm 1) = \p_{3}^2 u_{30}^- (x_2, \pm 1) = \p_{3} (u_{10}^- , u_{20}^-, P_0^- ) (x_2, \pm 1) = 0, \q \forall x_2 \in [-1, 1].
\end{cases}
\end{equation}

On the nozzle wall, the flow satisfies  the slip boundary condition
\begin{equation}\label{1-13}
\begin{cases}
u_2(x_1,\pm 1,x_3) =0,\ \   \forall (x_1, x_3) \in [L_0, L_1] \times [-1, 1],\\
u_3(x_1,x_2,\pm 1) =0,\ \  \forall (x_1, x_2) \in [L_0, L_1] \times [-1, 1].
\end{cases}
\end{equation}
Then the well-posedness    of the system \eqref{1-1} with \eqref{1-12} and \eqref{1-13} for supersonic flows can be solved by
 the classical theory developed  in \cite{BS07}. More precisely, there exists a constant $\eps_0 > 0$ depending only on the boundary data, such that for any $0 < \eps < \eps_0$, the system \eqref{1-1}  with \eqref{1-12} and \eqref{1-13} has a unique $C^{2, \al} (\overline{\mc{N}})$ solution $\mathbf\Psi^- = (u_1^-, u_2^-, u_3^-, P^-) (x_1, x_2, x_3)$, which satisfies
\begin{equation}\no
\| (u_1^-, u_2^-, u_3^-, P^-) - (\bar{u}^-, 0, 0, \bar{P}^-) \|_{C^{2, \al} (\overline{\mc{N}})} \le C_0 \eps,
\end{equation}
and
\begin{equation}\label{bc14}
\begin{cases}
  (u_2^-, \p_2^2 u_2^-) (x_1, \pm 1, x_3) = \p_2 (u_1^-, u_3^-, P^-) (x_1, \pm 1, x_3) = 0, & \forall (x_1, x_3) \in [L_0, L_1] \times [-1, 1], \\
  (u_3^-, \p_3^2 u_3^-) (x_1, x_2, \pm 1) = \p_3 (u_1^-, u_2^-, P^-) (x_1, x_2, \pm 1) = 0, & \forall (x_1, x_2) \in [L_0, L_1] \times [-1, 1].
\end{cases}
\end{equation}
\par At the exit of the nozzle, the end pressure is prescribed by
\begin{equation}\label{1-14}
P(L_1,x_2,x_3)=P_e+\epsilon P_{ex}(x_2,x_3), \ \ (x_2, x_3) \in E,
\end{equation}
 where 
  $ P_{ex} (x_2,x_3)\in C^{2,\alpha}(\overline{ E}) $ satisfies the following compatibility conditions
 \begin{equation}\label{1-15}
\begin{cases}
\p_{x_2} P_{ex}(\pm 1,x_3)= 0,\ \  x_3\in [-1,1],\\
\p_{x_3} P_{ex}(x_2,\pm 1) =0,\ \ x_2\in [-1,1].
\end{cases}
\end{equation}
	
 \par In this paper, we want to look for a piecewise smooth solution $ \mathbf{\Psi} $ which jumps only at a shock front $ \ms=\{(x_1,x_2,x_3):x_1=\xi(x_2,x_3), x' = (x_2,x_3)\in E\}$, has the following form:
\begin{equation}\no
\mathbf{\Psi}=
\begin{cases}
\mathbf{\Psi}^-:=(u_1^-,u_2^-,u_3^-,P^-)(x_1,x'),\quad {\rm{if}}\ L_0<x_1<\xi(x'), \ x' \in E,\\
\mathbf{\Psi}^+:=(u_1^+,u_2^+,u_3^+,P^+)(x_1, x'),\quad {\rm{if}}\ \xi(x')<x_1<L_1,\  x' \in E,
\end{cases}
\end{equation}
  and satisfies the following  Rankine-Hugoniot conditions on  the shock surface $ \ms $:
\begin{equation}\label{1-18}
\begin{cases}
[\rho u_1]-\p_{x_2}\xi[\rho u_2]-\p_{x_3}\xi[\rho u_3]=0,\\
[\rho u_1^2+P]-\p_{x_2}\xi[\rho u_1u_2]-\p_{x_3}\xi[\rho u_1u_3]=0,\\
[\rho u_1u_2]-\p_{x_2}\xi[\rho u_2^2+P]-\p_{x_3}\xi[\rho u_2u_3]=0,\\
[\rho u_1u_3]-\p_{x_2}\xi[\rho u_2 u_3]-\p_{x_3}\xi[\rho u_3^2+P]=0.
\end{cases}
\end{equation}

\begin{theorem}\label{main}
  Assume that the compatibility conditions  \eqref{bc13} and \eqref{1-15}  hold. There exists a suitable constant $\eps_0 > 0$ depending only on the background solution $\mathbf{\bar\Psi}$ and the boundary data $u_{10}^-$, $u_{20}^-$, $u_{30}^-$, $P_0^-$ $P_{ex}$ such that if $0 < \eps < \eps_0$, the problem \eqref{1-1} with \eqref{1-12}, \eqref{1-13}, \eqref{1-14} and \eqref{1-18} has a unique solution $\M\Psi^+ = (u_1^+, u_2^+, u_3^+, P^+) (x) $ with the shock front $\mc{S}: x_1 = \xi (x_2, x_3)$ satisfying the following properties.
  \begin{enumerate}[(i)]
    \item The function $\xi (x_2, x_3) \in C^{3, \al} (\overline{E})$ satisfies
        \begin{equation}\no
        \| \xi (x_2, x_3) - L_s \|_{C^{3, \al} (\overline{E})} \le C_* \eps,
        \end{equation}
        and
        \begin{equation}\no
        \begin{cases}
          \p_2 \xi (\pm 1, x_3)  = \p_2^3 \xi (\pm 1, x_3)  = 0, & \forall x_3 \in [-1, 1], \\
          \p_3 \xi (x_2, \pm 1)  = \p_3^3 \xi (x_2, \pm 1) = 0, & \forall x_2 \in [-1, 1],
        \end{cases}
        \end{equation}
        where the positive constant $C_* $ depends only on the background solution, the supersonic incoming flow and the exit pressure.
    \item The solution $\M\Psi^+ = (u_1^+, u_2^+, u_3^+, P^+) (x) \in C^{2, \al } (\overline{\mc{N}^+})$ satisfies the entropy condition
        \begin{equation}\no
        P^{+}(\xi(x_2,x_3),x_2,x_3)>P^{-}(\xi(x_2,x_3),x_2,x_3)\quad {\rm{on}}\ x_1=\xi(x_2,x_3), \q \forall (x_2, x_3) \in E
        \end{equation}
        and the estimate
        \begin{equation}\no
        \| \M\Psi^+ - \bar{\M\Psi}^+ \|_{C^{2, \al } (\overline{\mc{N}^+})} \le C_* \eps,
        \end{equation}
        with the compatibility conditions
        \begin{equation}\no
        \begin{cases}
          (u_2^+, \p_2^2 u_2^+) (x_1, \pm 1, x_3) = \p_2 (u_1^+, u_3^+, P^+) (x_1, \pm 1, x_3) = 0, & \forall (x_1, x_3) \in [\xi (x_2, x_3), L_1] \times [-1, 1], \\
          (u_3^+, \p_3^2 u_3^+) (x_1, x_2, \pm 1) = \p_3 (u_1^+, u_2^+, P^+) (x_1, x_2, \pm 1) = 0, & \forall (x_1, x_2) \in [\xi (x_2, x_3), L_1] \times [-1, 1].
        \end{cases}
        \end{equation}
  \end{enumerate}
\end{theorem}
\par We comment on the key ingredients of our mathematical analysis for Theorem \ref{main}. The transonic shock problem is reduced to a free boundary problem in a subsonic region, where the unknown shock surface is a part of the boundary and should be determined with the subsonic flow simultaneously.
Inspired by \cite{WX23}, we apply the deformation-curl decomposition introduced in \cite{WX19,WS19} to effectively decouple the hyperbolic and elliptic modes and reformulate the transonic shock  problem as a deformation-curl first order system for the velocity field with nonlocal terms, an algebraic equation to determine the shock front and two transport equations for the Bernoulli's quantity and the first component of the vorticity. Then by homogenizing the curl system and introducing a potential function, the deformation-curl elliptic system is reduced to a second order elliptic equation with a nonlocal term involving only the trace of the potential function on the shock front so that its unique solvability can be obtained.
\par This paper is organized as follows. In section \ref{reformulation}, we introduce the deformation-curl decomposition and reformulate the Rankine-Hugoniot jump conditions, and then reformulate the transonic shock problem in new coordinates. In section \ref{iteration}, we design an iteration scheme to prove Theorem \ref{main}. Some explicit expressions are given in Appendix \ref{appendix}.

\section{The reformulation of the transonic shock problem}\label{reformulation}\noindent

\par In this section, we first rewrite \eqref{1-1} using the deformation-curl decomposition, and reformulate the Rankine-Hugoniot conditions and boundary conditions, and finally reduce the transonic shock problem to a fixed boundary value problem by using appropriate coordinate transformation.
\subsection{The deformation-curl decomposition to the steady Euler system}\noindent

\par First, we study the hyperbolic modes in \eqref{1-1}. Denote the Bernoulli's function $ B $ by
\begin{equation}\label{B}
  B=\frac{1}{2}{|\textbf{u}|^2} +\frac{\gamma P}{(\gamma  - 1)\rho}-\Phi.
\end{equation}
Then 
$B$ satisfies the following transport equation
\begin{equation}\label{euler2}
   (u_1 \p_1 + u_2 \p_2 + u_3 \p_3) B = 0.
\end{equation}
Define the vorticity $\boldsymbol{\om} = \curl \u = ( \om_1, \om_2, \om_3)$ with
\begin{equation*}
\om_1 = \p_2 u_3 - \p_3 u_2, \, \om_2 = \p_3 u_1 - \p_1 u_3, \, \om_3 = \p_1 u_2 - \p_2 u_1.
\end{equation*}
It then follows from the third and fourth equations in \eqref{1-1} that
\begin{equation}\no
 \begin{cases}
    u_1 \om_3 - u_3 \om_1 + \p_2 B  = 0, \\
    u_2 \om_1 - u_1 \om_2 + \p_3 B = 0.
  \end{cases}
\end{equation}
Thus one  obtains
\begin{equation}\label{om2}
    \om_2 = \f {u_2 \om_1 + \p_3 B}{u_1},\ \ \om_3 = \f {u_3 \om_1 - \p_2 B}{u_1} .
\end{equation}
 Note that
\begin{equation}\label{om0}
\div \boldsymbol{\om} = \div \curl \u = 0.
\end{equation}
Substituting \eqref{om2} into \eqref{om0} yields that
\begin{equation}\label{om3}
 \b( \p_1 + \f {u_2}{u_1} \p_2 + \f {u_3}{u_1} \p_3 \b) \om_1
+ \b( \p_2 \b( \f {u_2}{u_1}\b) + \p_3 \b( \f {u_3}{u_1} \b) \b) \om_1
+  \p_2 \b(\f 1{u_1} \b) \p_3 B
- \p_3 \b( \f 1{u_1} \b) \p_2 B  = 0.
\end{equation}

Next, we analysis the elliptic modes in \eqref{1-1}. The equation \eqref{B} implies that 
\begin{equation}\label{rho1}
\rho = \rho (B, |\u|^2, \Phi) = \b( \f {\ga -1}{\ga } \b)^{\f 1 {\ga - 1}} \b( B - \f 12 |\u|^2 + \Phi \b)^{\f 1 {\ga -1}}.
\end{equation}
Substituting \eqref{rho1} into the continuity equation and combining with \eqref{euler2} lead to
\begin{equation}\label{rho2}
\begin{aligned}
&
\sum_{i=1}^3 (c^2 (B, |\u|^2, \Phi ) - u_i^2 ) \p_i u_i = u_1 (u_2 \p_1 u_2 + u_3 \p_1 u_3 - \p_1 \Phi ) \\
&\quad+ u_2 (u_1 \p_2 u_1 + u_3 \p_2 u_3 - \p_2 \Phi )
+ u_3 (u_1 \p_3 u_1 + u_2 \p_3 u_2 - \p_3 \Phi ).
\end{aligned}
\end{equation}
Together with the vorticity equations, this gives a deformation-curl system
\begin{equation}\label{dc}
\begin{cases}
 \sum_{i=1}^3 (c^2 (B, |\u|^2, \Phi ) - u_i^2 ) \p_i u_i = u_1 (u_2 \p_1 u_2 + u_3 \p_1 u_3 - \p_1 \Phi ) \\
 \quad\quad+ u_2 (u_1 \p_2 u_1 + u_3 \p_2 u_3 - \p_2 \Phi )+ u_3 (u_1 \p_3 u_1 + u_2 \p_3 u_2 - \p_3 \Phi ), \\
  \p_{x_2 } u_3 - \p_{x_3} u_2 = \om_1, \\
  \p_{x_3} u_1 - \p_{x_1} u_3 = \om_2, \\
  \p_{x_1} u_2 - \p_{x_2} u_1 = \om_3.
\end{cases}
\end{equation}
The equivalence of the system \eqref{dc} and an enlarged system with an additional unknown function is shown in Section \ref{iteration} (see equation \eqref{pf6}). This enlarged system in the subsonic region is elliptic in the sense of Agmon-Dougalis-Nirenberg \cite{ADN64}, which is verified in \cite{WS19}. 

\begin{lemma}\label{equiv1}
  Assume that $C^1$ smooth vector functions $(\rho, \u)$ is defined on a domain $\mc{N}$ which excludes the vacuum (i.e. $\rho (x) > 0$) and 
  $u_1 >0$ in $\mc{N}$. Then the following two statements are equivalent.
  \begin{enumerate}[(i)]
    \item $(\rho, {\bf u})$ satisfy the system \eqref{1-1} in $\mc{N}$;
    \item $({\bf u}, B)$ satisfy the equations \eqref{euler2}, \eqref{om2} and \eqref{dc}.
  \end{enumerate}
\end{lemma}

\subsection{The reformulation of the Rankine-Hugoniot conditions and boundary conditions}\noindent
\par The steady Euler system with an external force is elliptic-hyperbolic mixed in the subsonic region, which requires careful identification of  suitable boundary conditions and their compatibility. Define
\begin{equation*}
\begin{aligned}
&
w_1 (x_1, x^\prime) = u_1 (x_1, x^\prime) - \bar{u}^+ (x_1), \q
w_j (x_1, x^\prime) = u_j (x_1, x^\prime), \, j = 2, 3,\\
&
w_4 (x_1, x^\prime) = B (x_1, x^\prime) - \bar{B}^+, \q
{\bf w} = (w_1, \cdots, w_4), \\
&
w_5 (x^\prime) = \xi (x^\prime) - L_s, \quad x^\prime=(x_2,x_3).
\end{aligned}
\end{equation*}
Then the density and the pressure can be represented as
\be\label{rho3}
\rho (x_1, x^\prime) = \rho ({\bf w}) =
\b( \f {\ga -1}{\ga } \b)^{\f 1 {\ga - 1}} \b( w_4 + \bar{B} - \f 12 (w_1 + \bar{u})^2 - \f 12 \sum_{j = 2}^3 |w_j|^2 + \eps \Phi_0 + \bar{\Phi} \b)^{\f 1 {\ga -1}}, \\\label{rho4}
P (x_1, x^\prime) = P ({\bf w}) =
\b( \f {\ga -1}{\ga}\b)^{\f {\ga} {\ga - 1}}  \b( w_4 + \bar{B} - \f 12 (w_1 + \bar{u})^2 - \f 12 \sum_{j = 2}^3 |w_j|^2 + \eps \Phi_0 + \bar{\Phi} \b)^{\f {\ga} {\ga -1}}.
\ee

We now linearize the shock front. 
It follows from the third and fourth equations in \eqref{1-18} that
\begin{equation}\label{rh1}
 \p_2 \xi = \f {J_2 (\xi, x')}{J (\xi, x')}, \quad
 \p_3 \xi = \f {J_3 (\xi, x')}{J (\xi, x')},
\end{equation}
where
\begin{equation*}
\begin{aligned}
& J (\xi, x') = [\rho u_2^2 + P] [\rho u_3^2 + P ] - ([\rho u_2 u_3])^2, \\
& J_2 (\xi, x') = [\rho u_3^2 + P ] [\rho u_1 u_2 ] - [\rho u_1 u_3] [\rho u_2 u_3], \\
& J_3 (\xi, x') = [\rho u_2^2 + P] [\rho u_1 u_3] - [\rho u_1 u_2] [\rho u_2 u_3].
\end{aligned}
\end{equation*}
Then \eqref{rh1} can be rewritten as
\begin{equation}\label{rh2}
 \begin{cases}
    \p_2 \xi (x') = a_0 w_2 (\xi (x'), x') + g_2 (\M\Psi^- (L_s + w_5, x') - \bar{\M\Psi}^- (L_s + w_5), {\bf w} (\xi, x'), w_5), \\
    \p_3 \xi (x') = a_0 w_3 (\xi (x'), x') + g_3 (\M\Psi^- (L_s + w_5, x') - \bar{\M\Psi}^- (L_s + w_5), {\bf w} (\xi, x'), w_5),
  \end{cases}
\end{equation}
where $a_0 = \f {(\bar{\rho}^+ \bar{u}^+ ) (L_s)}{[\bar{P} (L_s)]} > 0$ and
\begin{equation*}
\begin{aligned}
&
g_2 (\M\Psi^- (L_s + w_5, x') - \bar{\M\Psi}^- (L_s + w_5), {\bf w} (\xi, x'), w_5) = \f {J_2}{J} - a_0 w_2 (\xi (x'), x'), \\
&
g_3 (\M\Psi^- (L_s + w_5, x') - \bar{\M\Psi}^- (L_s + w_5), {\bf w} (\xi, x'), w_5) = \f {J_3}{J} - a_0 w_3 (\xi (x'), x').
\end{aligned}
\end{equation*}
One can regard the functions $g_i$, $i = 2, 3$ as error terms which are bounded by
\begin{equation}\label{rh3}
|g_i| \le C_* (|\M\Psi^- (L_s + w_5, x') - \bar{\M\Psi }^- (L_s + w_5)| + |{\bf w} (\xi, x')|^2 + |w_5|^2).
\end{equation}
To obtain \eqref{rh3}, we only estimate $ g_2 $. The estimate of $ g_3 $ can be treated in the same way.     In fact,
\begin{equation*}
\begin{aligned}
g_2
 = &
\b( \f {\rho^+ u_1^+}{[P]} - a_0 \b) w_2
- \f {\rho^+ u_1^+ w_2}{[P]} \f {[\rho u_3^2]}{[\rho u_2^2 + P ]}
- \f {\rho^- u_1^- u_2^-}{[\rho u_2^2 + P ]}
 \\
&
+ \f { [\rho u_1 u_2]}{[\rho u_2^2 + P ]} \f {( [\rho u_2 u_3])^2  }{[\rho u_2^2 + P] [\rho u_3^2 + P ] - ([\rho u_2 u_3])^2}  - \f {[\rho u_1 u_3] [\rho u_2 u_3]}{J},
\end{aligned}
\end{equation*}
and
\begin{equation*}
\begin{aligned}
\b( \f {\rho^+ u_1^+}{[P]} - a_0 \b) w_2
= & \f {(\bar{\rho}^+ (\xi) + \hat{\rho}) (\bar{u}^+ (\xi) + w_1)}{P^+ (\xi) - \bar{P}^+ (\xi) + \bar{P}^+ (\xi) - \bar{P}^- (\xi) + \bar{P}^- (\xi) - P^- (\xi)} w_2 - \f {\bar{\rho}^+ (L_s) \bar{u}^+ (L_s) }{\bar{P}^+ (L_s) - \bar{P}^- (L_s)} w_2\\
= &
\b( \f {(\bar{\rho}^+ \bar{u}^+) (w_5 + L_s)}{ ( \bar{P}^+ - \bar{P}^-) (w_5 + L_s)} - \f {(\bar{\rho}^+ \bar{u}^+ ) (L_s) }{ (\bar{P}^+ - \bar{P}^- ) (L_s)} \b) w_2 \\
&
- \f {(\bar{\rho}^+ \bar{u}^+) (\xi)}{ ( \bar{P}^+ - \bar{P}^-) (\xi)}
\f {P^+ (\xi) - \bar{P}^+ (\xi) - (P^- (\xi) - \bar{P}^- (\xi))}{P^+ (\xi) - \bar{P}^+ (\xi) + \bar{P}^+ (\xi) - \bar{P}^- (\xi) + \bar{P}^- (\xi) - P^- (\xi)} w_2 \\
&
+ \f { \hat{\rho} \bar{u}^+ (\xi) + \bar{\rho}^+ (\xi) w_1 + \hat{\rho} w_1 }{P^+ (\xi) - \bar{P}^+ (\xi) + \bar{P}^+ (\xi) - \bar{P}^- (\xi) + \bar{P}^- (\xi) - P^- (\xi)} w_2.
\end{aligned}\end{equation*}
Thus \eqref{rh3} can be derived for $i =2, 3$.

It follows from \eqref{1-18} and \eqref{rh1} that
\begin{equation}\label{rh4}
  \begin{cases}
    [\rho u_1] = \f {[\rho u_2] J_2 + [\rho u_3] J_3}{J},  \\
    [\rho u_1^2 + P ] = \f {[\rho u_1 u_2] J_2 + [\rho u_1 u_3] J_3}{J}.
  \end{cases}
\end{equation}
Denote $ w_0 (x_1,x^\prime) =\rho(x_1,x^\prime)-\bar \rho^+(x_1) $. Then \eqref{rh4} implies that at $ (\xi,x^\prime) $, there holds
\begin{equation}\label{rh5}
\begin{cases}
 \bar u^+(L_s)w_0+\bar \rho^+(L_s)w_1=R_{01} (\M\Psi^- (\xi, x') - \bar{\M\Psi}^- (\xi), \M w (\xi, x') , w_5), \\
\left((\bar u^+(L_s))^2+c^2( \bar \rho^+(L_s))\right)w_0+2(\bar \rho^+\bar u^+)(L_s)w_1=-((\bar \rho^+-\bar \rho^-)\bar f)(L_s)w_5\\
\qquad+R_{02} (\M\Psi^- (\xi, x') - \bar{\M\Psi}^- (\xi), \M w (\xi, x'), w_5), \\
\end{cases}
\end{equation}
where
\begin{equation*}
\begin{aligned}
R_{01}& =- [\bar{\rho} \bar{u}] (\xi) + \f {[\rho u_2] J_2 + [\rho u_3] J_3}{J}+(\rho^- u_1^-) (\xi, x^\prime) - (\bar{\rho}^- \bar{u}^-) (\xi)\\
 &\quad-(w_1+\bar u^+(L_s+w_5)-\bar u^+(L_s))w_0(\xi, x^\prime)-(\bar\rho^+(L_s+w_5)-\bar \rho^+(L_s))w_1(\xi, x^\prime), \\
R_{02}& =-\left([\bar{\rho} \bar{u}^2+\bar P] (L_s+w_5)-((\bar \rho^+-\bar \rho^-)\bar f)(L_s)w_5\right)+(\rho^- (u_1^-)^2+ P^-) (\xi, x^\prime) \\
&\quad
- (\bar{\rho}^- (\bar{u}^-)^2+ \bar P^-) (\xi)
-\b((\rho^+ (u_1^+)^2+ P^+) (\xi, x^\prime) - (\bar{\rho}^+ (\bar{u}^+)^2+ \bar P^+) (\xi)\\
&\quad-\left((\bar u^+(L_s))^2+c^2( \bar \rho^+(L_s))\right)w_0+2(\bar \rho^+\bar u^+)(L_s)w_1\b)+\f {[\rho u_1 u_2] J_2 + [\rho u_1 u_3] J_3}{J}.
  \end{aligned}
\end{equation*}
\par Note that
\begin{equation*}
\frac{d}{d x_1}( \bar \rho \bar u)(x_1)=0, \quad \frac{d}{d x_1}(\bar \rho \bar u^2+ \bar P)(x_1)=(\bar \rho\bar f)(x_1).
\end{equation*}
Then
\begin{equation*}
 [\bar \rho \bar u](L_s+w_5)=O(w_5^2),\quad [\bar \rho \bar u^2+ \bar P]
 (L_s+w_5)-((\bar \rho^+-\bar \rho^-)\bar f)(L_s)w_5=O(w_5^2),
 \end{equation*}
and thus
\begin{equation}\no
|R_{0i}| \le C_0 (|\M\Psi^- (\xi , x') - \bar{\M\Psi}^- (\xi  )| + |\M w (\xi, x')|^2 + |w_5 (x' )|^2), \ \ i = 1, 2.
\end{equation}
\par By solving the algebraic equations \eqref{rh5}, one derives
\begin{equation}\label{rh10}
\begin{cases}
  w_0(\xi, x')  = b_0 w_5 (x') + R_0 (\M\Psi^- (\xi, x') - \bar{\M\Psi}^- (\xi), \M\Psi^+ (\xi, x') - \bar{\M\Psi}^+ (L_s), w_5),\\
   w_1 (\xi, x')  = b_1 w_5 (x') + R_1 (\M\Psi^- (\xi, x') - \bar{\M\Psi}^- (\xi), \M\Psi^+ (\xi, x') - \bar{\M\Psi}^+ (L_s), w_5), \\
  \end{cases}
\end{equation}
where
\begin{equation*}
b_0  =  - \f {(\bar\rho^+(L_s)-\bar\rho^-(L_s))\bar{f} (L_s) }
{ c^2 (\bar{\rho}^+ (L_s)) -  (\bar{u}^+ (L_s))^2   }  <0, \quad
b_1  = \f {\bar{u}^+ (L_s)(\bar\rho^+(L_s)-\bar\rho^-(L_s))\bar{f} (L_s) }
{ \bar{\rho}^+ (L_s)( c^2 (\bar{\rho}^+ (L_s)) - (\bar{u}^+ (L_s))^2)  }  >0,
\end{equation*}
and
\begin{equation*}
\begin{aligned}
R_0  = & \f {-2\bar{u}^+ (L_s) R_{01} + R_{02}}{(c^2 (\bar{\rho}^+ (L_s)) - ( \bar{u}^+ (L_s))^2)} \eqq \sum_{i = 1}^2 b_{0i} R_{0i}, \\
R_1  = & \f {\left((\bar u^+(L_s))^2+c^2( \bar \rho^+(L_s))\right) R_{01} - \bar{u}^+ (L_s)R_{02}}
{ \bar{\rho}^+(L_s) ( c^2 (\bar{\rho}^+ (L_s)) - (\bar{u}^+ (L_s))^2)  } \eqq \sum_{i = 1}^2 b_{1i} R_{0i}.
\end{aligned}
\end{equation*}
Next, it follow from the  Bernoulli's function and \eqref{rh10} that
\begin{equation}\no
w_4 (\xi, x')  = b_2 w_5 (x') + R_2 (\M\Psi^- (\xi, x') - \bar{\M\Psi}^- (\xi), \M w (\xi, x'), w_5),
\end{equation}
where
\be\no
b_2  =  \f {(\bar\rho^-(L_s)-\bar\rho^+(L_s))\bar{f} (L_s) }{\bar\rho^+(L_s)}  < 0,
\ee
\begin{equation*}
\begin{aligned}
R_{03} & = (\bar u^+(L_s+w_5)-\bar u^+(L_s))w_1(\xi,x^\prime) -\frac{c^2(\bar\rho^+(L_s)}{\bar\rho^+(L_s)}
w_0(\xi,x^\prime) \\
& \q +
\frac{1}{2}\sum_{i=1}^3w_j^2(\xi,x^\prime) +\frac{\gamma}{\gamma-1}(\rho^+(\xi,x^\prime)^{\gamma - 1}
-(\bar\rho^+(\xi))^{\gamma - 1}) - \eps \Phi_0 (\xi, x'), \\
R_2&= \f {- \bar{u}^+ (L_s) R_{01} + R_{02}}{\bar{\rho}^+ (L_s)} + R_{03} \eqq \sum_{i = 1}^3 b_{2i} R_{0 i}.
\end{aligned}
\end{equation*}
\par In the following,  the superscript ``+" in $\bar{u}^+$, $\bar{P}^+$ and $\bar{B}^+$  will be ignored to simplify the notations. We turn to concern the  boundary conditions at the exit. It follows from the Bernoulli's function that
\begin{equation}\no
w_4 = \bar{u} w_1 + \f {1} {\bar{\rho}} (P - \bar{P}) + \f 12 \sum_{i = 1}^3 w_j^2 + E ({\bf w} (x_1, x')) ,
\end{equation}
where
\begin{equation*}
E ({\bf w} (x_1, x')) = \f {\ga }{\ga - 1} (P ({\bf w}))^{\f {\ga - 1}{\ga }} - \f {\ga }{\ga - 1} \bar{P}^{\f {\ga - 1}{\ga }} - \f 1{\bar{\rho} (x_1)} (P ({\bf w}) - \bar{P}) - \eps \Phi_0 .
\end{equation*}
This, together with \eqref{1-14} yields that
\begin{equation}\label{bc3}
w_1 (L_1, x') = \f {w_4 (L_1, x')}{\bar{u} (L_1)} - \f {\eps P_{ex} (x')}{(\bar{\rho} \bar{u}) (L_1)} - \f {1}{2 \bar{u} (L_1)} \sum_{i = 1}^3 w_i^2 (L_1, x') - \f {E ({\bf w} (L_1, x')} {\bar{u} (L_1)} .
\end{equation}
The boundary conditions of $w_2$ and $w_3$ on the nozzle walls are
\begin{equation}\label{bc15}
\begin{cases}
  w_2 (x_1, \pm 1, x_3) = 0, & \mbox{on } L_s + w_5 < x_1 < L_1, \, x_3 \in [-1,1], \\
  w_3 (x_1, x_2, \pm 1) = 0, & \mbox{on } L_s + w_5 < x_1 < L_1, \, x_2 \in [-1,1].
\end{cases}
\end{equation}

Finally, equations \eqref{euler2}, \eqref{om2}, \eqref{om3} and \eqref{rho2} can be expressed using $w_1, \cdots , w_4$. The equation for $w_4$ is
\begin{equation}\label{euler4}
 \b( \p_1 + \f {w_2}{\bar{u} + w_1} \p_2 + \f {w_3}{\bar{u} + w_1} \p_3 \b) w_4 = 0.
\end{equation}
The equations for the vorticity $\om$ are
\begin{equation}\label{om4}
\begin{aligned}
& \b( \p_1 + \f {w_2}{\bar{u} + w_1} \p_2 + \f {w_3}{\bar{u} + w_1} \p_3 \b) \om_1
+ \b( \p_2 \b( \f {w_2}{\bar{u} + w_1}\b) + \p_3 \b( \f {w_3}{\bar{u} + w_1} \b) \b) \om_1 \\
&
+  \p_2 \b(\f 1{\bar{u} + w_1} \b) \p_3 w_4
 - \p_3 \b( \f 1{\bar{u} + w_1} \b) \p_2 w_4 = 0,
 \end{aligned}
\end{equation}
and
\begin{equation}\label{om5}
    \om_2 = \f {w_2 \om_1 + \p_3 w_4}{\bar{u} + w_1} , \  \ \ \  \om_3 = \f {w_3 \om_1 - \p_2 w_4}{\bar{u} + w_1} .
\end{equation}
The equation for $w_1$ is obtained by the following calculations. Note that
\begin{equation*}
\begin{cases}
 (c^2 (\bar{B}, \bar{u}, \bar{\Phi}) - \bar{u}^2) \bar{u}' (x_1) =- \bar{u}\bar f , \\
 c^2 (B, |\u|, \Phi ) - u_1^2 - c^2 (\bar{B}, \bar{u}, \bar{\Phi}) + \bar{u}^2 =
(\ga - 1) ( w_4 + \eps \Phi_0)  - \f {\ga + 1}{2} w_1^2 - (\ga + 1) \bar{u} w_1
 - \f {\ga -1 }2 \sum_{i = 2}^3 w_i^2.
 \end{cases}
\end{equation*}
Then one can rewrite \eqref{rho2} as
\begin{equation}\label{om6}
 (1 - \bar{M}^2 (x_1)) \p_1 w_1 + \p_2 w_2 + \p_3 w_3 + \f {\bar f - (\ga+1) \bar u \bar u^\prime }{c^2 (\bar{\rho}) } w_1+\f{(\gamma-1)\bar u^\prime}{c^2 (\bar{\rho}) } w_4   =  F ({\bf w}),
\end{equation}
where
\begin{equation*}
\begin{aligned}
\bar{M}^2 (x_1) & =  \f {\bar{u}^2 }{c^2 (\bar{\rho}) } (x_1),\\
F ({\bf w}) =&
 \f {(\ga + 1) \bar{u}}{c^2 (\bar{\rho})} w_1 \p_1 w_1
+ \f {w_1 + \bar{u}}{c^2 (\bar{\rho})}
 (w_2 \p_1 w_2 + w_3 \p_1 w_3 - \eps \p_1 \Phi_0  ) \\
 &
+ \f {w_2}{c^2 (\bar{\rho})} ((w_1 + \bar{u}) \p_2 w_1 + w_3 \p_2 w_3 - \eps \p_2 \Phi_0 ) \\
&
+ \f {w_3}{c^2 (\bar{\rho})} ((w_1 + \bar{u}) \p_3 w_1 + w_2 \p_3 w_2 - \eps \p_3 \Phi_0 )
\\
&
- \f 1{c^2 (\bar{\rho})} \b( \eps (\ga - 1) \Phi_0  - \f {\ga + 1}{2} w_1^2
 - \f {\ga -1 }2 \sum_{i = 2}^3 w_i^2 \b)  \bar{u}' \\
&
- \f 1{c^2 (\bar{\rho})} \b( (\ga - 1) ( w_4 + \eps \Phi_0 ) - \f {\ga + 1}{2} w_1^2
 - \f {\ga -1 }2 \sum_{i = 2}^3 w_i^2 \b) \p_1 w_1 \\
&
- \f 1{c^2 (\bar{\rho})}  \b( (\ga - 1) (\bar{u} w_1 + w_4 + \eps \Phi_0) - \f {\ga - 1}2 (w_1^2 + w_3^2) - \f {\ga + 1}2 w_2^2 \b) \p_2 w_2  \\
&
- \f 1{c^2 (\bar{\rho})}  \b( (\ga - 1) (\bar{u} w_1 + w_4 + \eps \Phi_0) - \f {\ga - 1}2 (w_1^2 + w_2^2) - \f {\ga +1 }2 w_3^2 \b) \p_3 w_3.
\end{aligned}
\end{equation*}

Therefore, solving the problem \eqref{1-1} with \eqref{1-12}-\eqref{1-15} and \eqref{1-18} is equivalent to finding vector functions ${\bf w}$ belonging to $\mc{N}_{w_5} \co \{ (x_1, x_2, x_3) : L_s + w_5 < x_1 < L_1, (x_2, x_3) \in E \}$ and a function $w_5$ belonging to $E$, which satisfy \eqref{euler4}-\eqref{om6} with boundary conditions \eqref{rh2}, \eqref{rh10}, \eqref{bc3} and \eqref{bc15}.

\subsection{Fix the domain and the reformulation of the problem}\noindent
\par To deal with the free boundary value problem, it is convenient
to reduce it into a fixed boundary value problem by setting
\begin{equation}\no
 y_1 = \f {x_1 - \xi }{L_1 - \xi } (L_1 - L_s) + L_s = \f {x_1 - v_5 - L_s}{L_1 - v_5 - L_s} (L_1 - L_s) + L_s, \, y_2 = x_2, \, y_3 = x_3,
\end{equation}
where $ v_5 (x_2, x_3) = \xi (x_2, x_3) - L_s $. Then
\be\no
\begin{cases}
  x_1 = y_1 + \f {L_1 - y_1}{L_1 - L_s} v_5 \eqq   D_0^{v_5}, \quad
  &\p_1 = \f {L_1 - L_s}{L_1 - v_5 - L_s} \p_{y_1} \eqq   D_1^{v_5}, \\
  \p_2 = \p_{y_2} + \f {(y_1 - L_1) \p_{y_2} v_5}{L_1 - v_5 - L_s} \p_{y_1} \eqq D_2^{v_5}, \quad
  &\p_3 = \p_{y_3} + \f {(y_1 - L_1) \p_{y_3} v_5}{L_1 - v_5 - L_s} \p_{y_1} \eqq  D_3^{v_5},
\end{cases}
\ee
and the domain $\mc{N}^+$ becomes
$
\mb{D} = \{ (y_1, y') : y_1 \in (L_s, L_1), y' = (y_2, y_3) \in E \}.
$
Denote
\begin{equation*}
\begin{aligned}
& \Sigma_2^{\pm} = \{ (y_1, \pm1, y_3): (y_1, y_3) \in (L_s, L_1) \times (-1 , 1) \}, \\
&  \Sigma_3^{\pm} = \{ (y_1, y_2, \pm 1): (y_1, y_2) \in (L_s, L_1) \times (-1 , 1) \}.
\end{aligned}
\end{equation*}
\par Set
\begin{equation*}
v_j (y) = w_j (y_1 + \f {L_1 - y_1}{L_1 - L_s} v_5, y_2, y_3 ), \, j = 1, \cdots , 4, \quad
  \tilde{\om}_j (y) = \om_j (y_1 + \f {L_1 - y_1}{L_1 - L_s} v_5, y_2, y_3 ), \, j = 1, 2, 3.
\end{equation*}
Then $\rho (x_1, x_2, x_3)$ and $P (x_1, x_2, x_3)$ in \eqref{rho3} and \eqref{rho4} can be  rewritten as
\be \no
 \tilde{\rho} ({\bf v} (y), v_5) =
\b( \f {\ga -1}{\ga } \b)^{\f 1 {\ga - 1}} \b( v_4 + \bar{B} - \f 12 (v_1 + \bar{u} (D_0^{v_5}))^2 - \f 12 \sum_{j = 2}^3 |v_j|^2 + \eps \Phi_0 + \bar{\Phi} \b)^{\f 1 {\ga -1}},
\ee
\be\no
 \tilde{P} ({\bf v} (y), v_5)
=
\b( \f {\ga -1}{\ga}\b)^{\f {\ga} {\ga - 1}} \b( v_4 + \bar{B} - \f 12 (v_1 + \bar{u} (D_0^{v_5}))^2 - \f 12 \sum_{j = 2}^3 |v_j|^2 + \eps \Phi_0 + \bar{\Phi} \b)^{\f {\ga} {\ga -1}}.
\ee
 Furthermore, after the coordinate transformation, \eqref{rh2} becomes
\be\label{rh11}
&& \p_{y_2} v_5(y') = a_0 v_2 (L_s, y') + g_2 ({\bf v} (L_s, y'), v_5 (y')), \\\label{rh15}
&& \p_{y_3} v_5 (y') = a_0 v_3 (L_s, y') + g_3 ({\bf v} (L_s, y'), v_5 (y')),
\ee
where
\be\label{rh12}
&& g_2 ({\bf v} (L_s, y'), v_5 (y')) =  \f {J_2 ({\bf v} (L_s, y'), v_5 (y'))}{J ({\bf v} (L_s, y'), v_5 (y'))} - a_0 v_2 (L_s, y'),
\\\label{rh13}
&& g_3 ({\bf v} (L_s, y'), v_5 (y')) = \f {J_3 ({\bf v} (L_s, y'), v_5 (y'))}{J ({\bf v} (L_s, y'), v_5 (y'))} - a_0 v_3 (L_s, y'),
\ee
and the exact formulas for $J$, $J_2$, and $J_3$ are given in the Appendix \ref{appendix}, which will be required to obtain the compatibility conditions in the below section.
\par By \eqref{rh10}, the shock front will be determined as follows
\begin{equation}\label{rh14}
v_5 (y') = \f 1{b_1} v_1 (L_s, y') - \f 1 {b_1} R_1 ({\bf v} (L_s, y'), v_5 (y')),
\end{equation}
where $R_1 ({\bf v} (L_s, y'), v_5 (y')) = \sum_{i = 1}^2  b_{1i} R_{0i} ({\bf v} (L_s, y'), v_5 (y'))$ and the exact formulas for $R_{0i}$, $i = 1, 2$ in $y-$coordinates are given in the Appendix \ref{appendix}. We will verify the compatibility conditions (See \eqref{bc26}-\eqref{bc27} below) using these exact formulas.

 The function $v_4$ will be determined by the second equation in \eqref{rh10}.  That is
\begin{equation}\label{euler6}
\begin{cases}
  \b( D_1^{v_5} + \f {v_2}{\bar{u} (D_0^{v_5}) + v_1} D_2^{v_5} + \f {v_3}{\bar{u} (D_0^{v_5}) + v_1} D_3^{v_5} \b) v_4 = 0, \\
  v_4 (L_s, y') =  b_2 v_5 (y') + R_2 ({\bf v} (L_s, y'), v_5 (y')),
\end{cases}
\end{equation}
where $R_2 ({\bf v} (L_s, y'), v_5 (y')) = \sum_{i = 1}^3 b_{2i} R_{0i} ({\bf v} (L_s, y'), v_5 (y'))$ and the exact formulas for $R_{0i}$, $i = 1, 2, 3$ can be found in the Appendix \ref{appendix}.

Note that \eqref{rh11} and \eqref{rh15} are equivalent to
\begin{equation}\no
\begin{cases}
 F_2 (y') = \p_{y_2} v_5 (y') - a_0 v_2 (L_s, y') - g_2 ({\bf v} (L_s, y'), v_5 (y')), \q \forall y' \in E, \\
 F_3 (y') = \p_{y_3} v_5 (y') - a_0 v_3 (L_s, y') - g_3 ({\bf v} (L_s, y'), v_5 (y')), \q \forall y' \in E.
 \end{cases}
\end{equation}
This following reformulation is essential for us to solve the transonic shock problem.
\begin{lemma}\label{lm1}
  Let $F_j$, $j = 2, 3$ be two $C^1$ smooth functions defined on $\overline{E}$. Then the following
two statements are equivalent.
  \begin{enumerate}[(i)]
    \item $F_2 = F_3 \equiv 0 $ on $\overline{E}$;
    \item $F_2$ and $F_3$ solve the following problem
        \begin{equation}\label{euler8}
        \begin{cases}
          \p_{y_2} F_3 - \p_{y_3} F_2 = 0, \q \text{in } E, \\
          \p_{y_2} F_2 + \p_{y_3} F_3 = 0, \q \text{in } E, \\
          F_2 (\pm 1 , y_3)  = 0, \q \text{on } y_3 \in [-1,1], \\
          F_3 (y_2, \pm 1 ) = 0, \q \text{on } y_2 \in [-1,1].
        \end{cases}
        \end{equation}
  \end{enumerate}
\end{lemma}
\par The first equation in \eqref{euler8} implies that
\begin{equation}\label{rh18}
\p_{y_3} v_2 (L_s, y') - \p_{y_2} v_3 (L_s, y') = \f 1{a_0} \b( \p_{y_2} \{g_3 ({\bf v} (L_s, y'), v_5 (y')) \}- \p_{y_3} \{g_2 ({\bf v} (L_s, y'), v_5 (y'))\} \b),
\end{equation}
which yields the boundary condition for the first component of the vorticity on the shock front.

The second equation in \eqref{euler8} yields that
\begin{equation}\label{rh19}
\p_{y_2}^2 v_5 (y') + \p_{y_3}^2 v_5 (y') - a_0 \p_{y_2} v_2(L_s, y') - a_0 \p_{y_3} v_3 (L_s, y')
 = \sum_{i=2}^3\p_{y_i}\{g_i ({\bf v} (L_s, y'), v_5(y'))\}.
\end{equation}
This, together with \eqref{rh14}, gives
\begin{equation}\label{rh20}
\p_{y_2}^2 v_1 (L_s, y') + \p_{y_3}^2 v_1 (L_s, y') - a_0 b_1 \p_{y_2} v_2 (L_s, y') - a_0 b_1 \p_{y_3} v_3 (L_s, y') = q_1 ({\bf v} (L_s, y'), v_5 (y')),
\end{equation}
where
\begin{equation*}
\begin{aligned}
q_1 ({\bf v} (L_s, y'), v_5 (y'))=b_1 \sum\limits_{i=2}^3\p_{y_i} \{g_i ({\bf v} (L_s, y'), v_5 (y'))\} +\sum\limits_{i=2}^3\p_{y_i}^2 \{R_1 ({\bf v} (L_s, y'), v_5 (y'))\}.
\end{aligned}
\end{equation*}
The equation \eqref{rh20} is the boundary condition for the deformation-curl system associated with the velocity field on the shock front.

We can express the boundary conditions in \eqref{euler8} as
\be\label{bc4}
 (\p_{y_2} v_1 - a_0 b_1 v_2) (L_s, \pm 1, y_3 ) = q_2^{\pm} ({\bf v} (L_s, \pm 1, y_3 ), v_5 (\pm 1 , y_3)), \q \forall y_3 \in [-1,1], \\\label{bc17}
 (\p_{y_3} v_1 - a_0 b_1 v_3) (L_s,  y_2, \pm 1 ) = q_3^{\pm} ({\bf v} (L_s,  y_2, \pm 1 ), v_5 (y_2, \pm 1)), \q \forall y_2 \in [-1,1],
\ee
with
\begin{equation*}
\begin{aligned}
& q_2^{\pm} ({\bf v} (L_s, \pm 1 , y_3 ), v_5 (\pm 1 , y_3))
= \p_{y_2} \{R_1 ({\bf v} (L_s, \cdot ), v_5 (\cdot))\} (\pm 1 , y_3) + b_1 g_2 ({\bf v} (L_s, \cdot ), v_5 (\cdot)) (\pm 1 , y_3), \\
&
q_3^{\pm} ({\bf v} (L_s, \pm 1, y_3 ), v_5 (\pm 1, y_3))
= \p_{y_3}\{ R_1 ({\bf v} (L_s, \cdot ), v_5 (\cdot))\} ( \pm 1, y_3) + b_1 g_3 ({\bf v} (L_s, \cdot ), v_5 (\cdot)) ( \pm 1, y_3).
\end{aligned}\end{equation*}

The vorticity can be determined as follows. \eqref{om4} is changed to be
\be\label{om7}
\b( D_1^{v_5} + \sum_{i=2}^3\f {v_i}{\bar{u} (D_0^{v_5}) + v_1} D_i^{v_5} \b) \tilde{\om}_1 + \sum_{i=2}^3\b( D^{v_5}_i \b( \f {v_i}{\bar{u} (D_0^{v_5}) + v_1}\b) \b)\om_1= H_0 ({\bf v}, v_5),
\ee
with
\begin{equation*}
 H_0 ({\bf v}, v_5) = D^{v_5}_3 \b( \f 1{\bar{u} (D_0^{v_5}) + v_1} \b) D^{v_5}_2 v_4
-  D^{v_5}_2 \b(\f 1{\bar{u} (D_0^{v_5}) + v_1} \b) D^{v_5}_3 v_4 .
\end{equation*}
The boundary data for $\tilde{\om}_1$ is given by \eqref{rh18} at $y_1 = L_s$, for $\forall y' \in E $,
\begin{equation}\label{bc6}
\tilde{\om}_1 (L_s, y') = \f 1{a_0} \b( \p_{y_2} \{g_3 ({\bf v} (L_s, y'), v_5 (y'))\} - \p_{y_3} \{g_2 ({\bf v} (L_s, y'), v_5 (y'))\} \b) + g_4 ({\bf v} (L_s, y'), v_5 (y')) ,
\end{equation}
where
\begin{equation}\label{bc7}
g_4 ({\bf v} (L_s, y'), v_5 (y'))
=
\f {(y_1 - L_1) \p_{y_2} v_5 }{L_1 - v_5  - L_s} \p_{y_1}  v_3 (L_s, y')
-  \f {(y_1 - L_1) \p_{y_3} v_5 }{L_1 - v_5 - L_s} \p_{y_1} v_2 (L_s, y').
\end{equation}
 Next, the equation \eqref{rh19} gives
\be\no
&& \tilde{\om}_2 = D_3^{v_5 } v_1 - D_1^{v_5} v_3  =
\f {v_2 \tilde{\om}_1 + D^{v_5}_3 v_4}{\bar{u}(D_0^{v_5}) + v_1}
, \\\no
&& \tilde{\om}_3 = D_1^{v_5} v_2 - D_2^{v_5} v_1  =
\f {v_3 \tilde{\om}_1 - D^{v_5}_2 v_4}{\bar{u} (D_0^{v_5}) + v_1}.
\ee
Then
\be\label{euler9}
&& \p_{y_2} v_3 - \p_{y_3} v_2 = \tilde{\om}_1 + H_1 ({\bf v}, v_5), \\\label{euler10}
&& \p_{y_3} v_1 - \p_{y_1} v_3 - \f 1{\bar{u} (y_1)} \p_{y_3} v_4
=
\f {v_2 \tilde{\om}_1 }{\bar{u}(D_0^{v_5}) + v_1} + H_2 ({\bf v}, v_5) , \\\label{euler11}
&& \p_{y_1} v_2 - \p_{y_2} v_1 + \f 1{\bar{u} (y_1)} \p_{y_2} v_4
= \f {v_3 \tilde{\om}_1 }{\bar{u} (D_0^{v_5}) + v_1} + H_3 ({\bf v}, v_5),
\ee
where
\begin{equation*}\begin{aligned}
H_1 ({\bf v}, v_5) =&
\f {(y_1 - L_1) (\p_{y_3} v_5 \p_{y_1} v_2 - \p_{y_2} v_5 \p_{y_1} v_3) }{L_1 - v_5 - L_s}, \\
H_2 ({\bf v}, v_5) =&
- \f {(y_1 - L_1) \p_{y_3} v_5 \p_{y_1} v_1}{L_1 - v_5 - L_s}
+
\f {v_5 \p_{y_1} v_3}{L_1 - v_5 - L_s}
+ \b( \f 1 {\bar{u} (D_0^{v_5}) + v_1} - \f 1{\bar{u} (y_1)} \b) \p_{y_3} v_4\\
& +
\f {(y_1 - L_1) \p_{y_3} v_5}{(\bar{u} (D_0^{v_5}) + v_1) (L_1 - v_5 - L_s )} \p_{y_1} v_4, \\
H_3 ({\bf v}, v_5) =&
- \f {(y_1 - L_1) \p_{y_2} v_5 \p_{y_1} v_1}{L_1 - v_5 - L_s}
-
\f {v_5 \p_{y_1} v_2}{L_1 - v_5 - L_s}
- \b( \f 1 {\bar{u} (D_0^{v_5}) + v_1} - \f 1{\bar{u} (y_1)} \b) \p_{y_2} v_4 \\
& - \f {(y_1 - L_1) \p_{y_2} v_5}{(\bar{u} (D_0^{v_5}) + v_1) (L_1 - v_5 - L_s )} \p_{y_1} v_4.
\end{aligned}\end{equation*}
The boundary condition \eqref{1-13} can be rewritten on $\Sigma_2^{\pm}$ and $\Sigma_3^{\pm}$ as
\begin{equation}\label{bc8}
\begin{cases}
v_2 (y_1, \pm 1, y_3) = 0, & \text{on } \Sigma_2^{\pm}, \\
   v_3 (y_1, y_2, \pm 1) = 0, & \text{on } \Sigma_3^{\pm}.
\end{cases}
\end{equation}
\par Moreover, the equation \eqref{om6} becomes
\begin{equation}\label{euler12}
d_1 (y_1) \p_{y_1} v_1 + \p_{y_2} v_2 + \p_{y_3} v_3 + d_2 (y_1) v_1+ d_3 (y_1) v_4 = \mc{G}_0 ({\bf v}, {v_5}),
\end{equation}
with
\begin{equation*}\begin{aligned}
 d_1 (y_1) = & 1 - \bar{M}^2 (y_1), \q
d_2 (y_1) = \f { \bar{f} (y_1)-(\ga+1)\bar u\bar u^\prime(y_1)}{c^2 (\bar{\rho}(y_1)) }, \q d_3 (y_1) = \f {(\ga - 1) \bar{u}' (y_1)}{c^2 (\bar{\rho}(y_1)) },  \\
 \mc{G}_0 ({\bf v}, {v_5}) = & \mc{F} ({\bf v}, {v_5}) - (d_1 (D_0^{v_5}) D_1^{v_5} v_1 - d_1 (y_1) \p_{y_1} v_1) - (D_2^{v_5} v_2 - \p_{y_2} v_2) \\
&  - (D_3^{v_5} v_3 - \p_{y_3} v_3) - (d_2 (D_0^{v_5}) v_1 - d_2 (y_1) v_1)- (d_3 (D_0^{v_5}) v_4  - d_3 (y_1) v_4),
\end{aligned}\end{equation*}
\begin{equation*}\begin{aligned}
 \mc{F} ({\bf v}, {v_5})
 = &
  \f {(\ga + 1) \bar{u} (D_0^{v_5})}{c^2 (D_0^{v_5})} v_1 D^{v_5}_1 v_1
+ \f {v_1 + \bar{u}}{c^2 (D_0^{v_5})}
 (v_2 D^{v_5}_1 v_2 + v_3 D^{v_5}_1 v_3 - \eps D^{v_5}_1 \Phi_0  ) \\
 &
+ \f {v_2}{c^2 (D_0^{v_5})} ((v_1 + \bar{u} (D_0^{v_5})) D^{v_5}_2 v_1 + v_3 D^{v_5}_2 v_3 - \eps D^{v_5}_2 \Phi_0 ) \\
&
+ \f {v_3}{c^2 (D_0^{v_5})} ((v_1 + \bar{u} (D_0^{v_5})) D^{v_5}_3 v_1 + v_2 D^{v_5}_3 v_2 - \eps D^{v_5}_3 \Phi_0 )
\\
&
- \f 1{c^2 (D_0^{v_5})} \b( \eps (\ga - 1)   \Phi_0   - \f {\ga + 1}{2} v_1^2
 - \f {\ga -1 }2 \sum_{i = 2}^3 v_i^2 \b)  \bar{u}' (D_0^{v_5})\\
&
- \f 1{c^2 (D_0^{v_5})} \b( (\ga - 1) ( v_4 + \eps \Phi_0)   - \f {\ga + 1}{2} v_1^2
 - \f {\ga -1 }2 \sum_{i = 2}^3 v_i^2 \b) D^{v_5}_1 v_1  \\
&
- \f 1{c^2 (D_0^{v_5})}  \b( (\ga - 1) (\bar{u}(D_0^{v_5}) v_1 + v_4 + \eps \Phi_0) - \f {\ga - 1}2 (v_1^2 + v_3^2) - \f {\ga + 1}2 v_2^2 \b) D^{v_5}_2 v_2  \\
&
- \f 1{c^2 (D_0^{v_5})}  \b( (\ga - 1) (\bar{u}(D_0^{v_5}) v_1 + v_4 + \eps \Phi_0) - \f {\ga - 1}2 (v_1^2 + v_2^2) - \f {\ga +1 }2 v_3^2 \b) D^{v_5}_3 v_3.
\end{aligned}
\end{equation*}

Finally, the boundary condition \eqref{bc3} at the exit becomes
\begin{equation}\label{bc9}
 v_1 (L_1, y')
- \f {v_4 (L_1, y')}{\bar{u} (L_1)} = - \f {\eps P_{ex} (y')}{(\bar{\rho} \bar{u})(L_1)}
- \f 1{2 \bar{u} (L_1)} \sum_{i = 1}^3 v_i^2 (L_1, y') - \f 1{\bar{u} (L_1)} E ({\bf w} (L_1, y')),
\end{equation}
where
\begin{equation}\label{bc10}
  E ({\bf v} (L_1, y')) = \f {\ga }{\ga - 1}  (P ({\bf v}) (L_1, y'))^{\f {\ga - 1}{\ga }} - \f {\ga }{\ga - 1} \bar{P}^{\f {\ga - 1}{\ga }}  - \f 1{\bar{\rho} (D_0^{v_5})} (P ({\bf v}) (L_1, y') - \bar{P })  - \eps \Phi_0.
\end{equation}

Hence, finding the solution for the system \eqref{1-1} with \eqref{1-12}-\eqref{1-15} and \eqref{1-18} is equivalent to solve the following problem.

\textbf{Problem TS.} Find a function $v_5$ defined on $E$ and vector function $(v_1, \cdots, v_4)$ defined on the $\mb{D}$, which solve the equations \eqref{euler6}, \eqref{om7}, \eqref{euler9}-\eqref{euler11} and \eqref{euler12} with the boundary conditions \eqref{rh20}-\eqref{bc17}, \eqref{bc6}, \eqref{bc8} and \eqref{bc9}.

\begin{theorem}\label{thm1}
  Assume that the compatibility conditions \eqref{bc13} and \eqref{1-15} hold. There exists a small constant $\eps_0 > 0$ depending only on the background solution $\bar{\M\Psi }$ and the boundary data $u_{1,0}^-$, $u_{2, 0}^-$, $u_{3, 0}^-$, $P_0^-$, $P_e$ such that if $0 \le \eps < \eps_0$, the problem \eqref{euler6}, \eqref{om7}, \eqref{euler9}-\eqref{euler11}, \eqref{euler12} with the boundary conditions \eqref{rh20}-\eqref{bc17}, \eqref{bc6}, \eqref{bc8} and \eqref{bc9} has a unique solution $(v_1, v_2, v_3, v_4) (y)$ with the shock front $\mc{S}: y_1 = v_5 (y')$ satisfying the following properties.
  \begin{enumerate}[(i)]
    \item The function $v_5 (y') \in C^{3, \al} (\overline{E})$ satisfies
    \begin{equation}\no
    \| v_5 (y')  \|_{C^{3, \al} (\overline{E})} \le C_* \eps,
    \end{equation}
    and
    \begin{equation}\no
    \begin{cases}
       \p_{y_2} v_5 (\pm 1, y_3) = \p_{y_2}^3 v_5 (\pm 1, y_3) = 0 , & \forall y_3 \in [-1 ,1], \\
       \p_{y_3} v_5 (y_2, \pm 1) = \p_{y_3}^3 v_5 (y_2, \pm 1) = 0 , & \forall y_2 \in [-1,1],
    \end{cases}
    \end{equation}
    where $C_*$ is a positive constant depending only on the background solution, the supersonic incoming flow and the exit pressure.
    \item The solution $( v_1, v_2, v_3, v_4) (y) \in C^{2, \al} (\overline{\mb{D}})$ satisfies the estimate
        \begin{equation}\no
        \sum_{i = 1}^4 \| v_i \|_{C^{2, \al} (\overline{\mb{D}})} \le C_* \eps,
        \end{equation}
        and the compatibility conditions
        \begin{equation}\no
        \begin{cases}
           (v_2, \p_{y_2}^2 v_2) (y_1, \pm 1, y_3) = \p_{y_2} (v_1, v_3, v_4) (y_1, \pm 1, y_3)  = 0 , & \mbox{on } \Sigma_2^{\pm}, \\
           (v_3, \p_{y_3}^2 v_3) (y_1, y_2, \pm 1) = \p_{y_3} (v_1, v_2, v_4) (y_1, y_2, \pm 1)  = 0 , & \mbox{on } \Sigma_3^{\pm}.
        \end{cases}
        \end{equation}
  \end{enumerate}
\end{theorem}

\section{Iteration scheme and the proof of Theorem \ref{thm1}}\label{iteration}\noindent
\par In this section, we design an iteration to prove Theorem \ref{thm1}. Define the solution class $\mc{V}$ consists of vectors $(v_1, \cdots, v_4, v_5) \in (C^{2, \al} (\overline{\mb{D}}))^4 \times C^{3, \al} (\overline{E})$ satisfying
\begin{equation}\no
\| ({\bf v}, v_5) \|_{\mc{V}} \co \sum_{i = 1}^4 \| v_i \|_{C^{2, \al} (\overline{\mb{D}})} + \| v_5 \|_{C^{3, \al} (\overline{E})} \le \de_0,
\end{equation}
and the compatibility conditions
\begin{equation}\label{bc20}
\begin{cases}
    (v_2, \p_{y_2}^2 v_2) (y_1, \pm 1, y_3) = \p_{y_2} (v_1, v_3, v_4) (y_1, \pm 1, y_3)  = 0 , & \mbox{on } \Sigma_2^{\pm},\\
   (v_3, \p_{y_3}^2 v_3) (y_1, y_2, \pm 1) = \p_{y_3} (v_1, v_2, v_4) (y_1, y_2, \pm 1)  = 0 , & \mbox{on } \Sigma_3^{\pm},\\
    (\p_{y_2} v_5, \p_{y_2}^3 v_5) (\pm 1, y_3) = 0, & \mbox{on } y_3 \in [-1 , 1], \\
   (\p_{y_3} v_5, \p_{y_3}^3 v_5) (y_2, \pm 1) = 0, & \mbox{on } y_2 \in [-1, 1].
\end{cases}
\end{equation}

Given any $( \hat{{\bf v}}, \hat{v}_5) \in \mc{V} $, we construct an iterative procedure that generates a new $({\bf v}, v_5) \in \mc{V}$, and thus we define a mapping $\mc{T}$ from $\mc{V}$ to itself by choosing a small positive constant $\de_0$. We first solve the transport equation for $v_4$ to obtain its expression. Similarly, we can solve the first component of the vorticity and get its compatibility conditions in terms of the trace $ v_1(L_s,y^\prime)$. Then we enlarge the deformation-curl system by introducing an additional unknown function with additional boundary conditions, and use the Lax-Milgram theorem and the fourier series expansion to obtain the existence and uniqueness of the velocity field. This enables us to uniquely determine $v_4$ and $v_5$.

\textbf{Step 1.}
Assume that $v_1 (L_s, y')$ is obtained, the shock front $v_5$ can be uniquely determined by the following algebraic equation
\begin{equation}\label{pf1}
v_5 (y') = \f 1{b_1} v_1 (L_s, y') - \f 1{b_1} R_1 (\hat{{\bf v}} (L_s, y'), \hat{v}_5).
\end{equation}

\textbf{Step 2.} Solving the transport equation for the Bernoulli's quantity.
\begin{equation}\label{pf2}
\begin{cases}
  \b( D_1^{\hat{v}_5} + \f {\hat{v}_2}{\bar{u} (D_0^{\hat{v}_5}) + \hat{v}_1} D_2^{\hat{v}_5} + \f {\hat{v}_3}{\bar{u} (D_0^{\hat{v}_5}) + \hat{v}_1} D_3^{\hat{v}_5} \b) v_4 = 0, \\
  v_4 (L_s, y') = b_2 v_5 (y') + R_2 (\hat{{\bf v}} (L_s, y'), \hat{v}_5).
\end{cases}
\end{equation}
Set
\be\label{I2}
I_2 (y) \co \f { (L_1 - v_5 - L_s) \hat{v}_2}{(L_1 - L_s + (y_1 - L_1) (\p_{y_2} v_5 + \p_{y_3} v_5)) ( \bar{u} (D_0^{\hat{v}_5}) + v_1 )} , \\\label{I3}
I_3 (y) \co \f { (L_1 - v_5 - L_s ) \hat{v}_3}{(L_1 - L_s + (y_1 - L_1) (\p_{y_2} v_5 + \p_{y_3} v_5)) ( \bar{u} (D_0^{\hat{v}_5}) + \hat{v}_1) }.
\ee
Then $I_2, I_3 \in C^{2, \al } (\overline{E})$ for any $({\bf \hat{v}}, \hat{v}_5) \in \mc{V}$. The function $v_4$ is conserved along the trajectory defined by the following ordinary differential equations
\begin{equation}\label{ode1}
\begin{cases}
  \f {d \bar{y}_2 (\tau; y)}{d \tau} = I_2 (\tau, \bar{y}_2 (\tau; y), \bar{y}_3 (\tau; y) ), & \forall \tau \in [L_s, L_1], \\
  \f {d \bar{y}_3 (\tau; y)}{d \tau} = I_3 (\tau, \bar{y}_2 (\tau; y), \bar{y}_3 (\tau; y) ), & \forall \tau \in [L_s, L_1], \\
  \bar{y}_2 (y_1; y) = y_2, \bar{y}_3 (y_1; y) = y_3.  &
\end{cases}
\end{equation}
Denote $(\beta_2 (y), \beta_3 (y)) = (\bar{y}_2 (L_s; y), \bar{y}_3 (L_s; y))$. Since $({\bf \hat{v}}, \hat{v}_5) \in \mc{V}$, it follows from \eqref{bc20} that
\begin{equation}\label{bc21}
\begin{cases}
I_2 (y_1, \pm 1, y_3) = \p_{y_2} I_3 (y_1, \pm 1, y_3) = 0, & \mbox{on } \ \Sigma_2^{\pm}, \\
   I_3 (y_1, y_2, \pm 1) = \p_{y_3} I_2 (y_1, y_2, \pm 1) = 0, & \mbox{on } \ \Sigma_3^{\pm}.
\end{cases}
\end{equation}
Due to the uniqueness of the solution of \eqref{ode1}-\eqref{bc21}, one obtains
\begin{equation}\label{bc22}
\begin{cases}
  \bar{y}_2 (\tau; y_1, \pm 1, y_3) = \pm 1, & \forall \tau \in [L_s, L_1], \, (y_1, y_3) \in \Sigma_2^{\pm}, \\
  \bar{y}_3 (\tau; y_1, y_2, \pm 1) = \pm 1, & \forall \tau \in [L_s, L_1], \, (y_1, y_2) \in \Sigma_3^{\pm},
\end{cases}
\end{equation}
and
\begin{equation}\label{bc23}
\begin{cases}
\beta_2 (y_1, \pm 1, y_3) = \pm 1, & \forall (y_1, y_3) \in \Sigma_2^{\pm }, \\
  \beta_3 (y_1, y_2, \pm 1) = \pm 1, & \forall (y_1, y_2) \in \Sigma_3^{\pm }.
\end{cases}
\end{equation}
The standard theory of systems of ordinary differential equations and \eqref{bc22} guarantee the existence and uniqueness of $(\bar{y}_2 (\tau;y), \bar{y}_3 (\tau; y))$ on the whole interval $[L_s, L_1]$. Then \eqref{ode1} implies
\begin{equation*}\begin{aligned}
& y_2 - \beta_2 (y) = \int_{L_s}^{y_1} I_2 (\tau, \bar{y}_2 (\tau; y), \bar{y}_3 (\tau; y) ) d \tau, \\
& \de_{2j} - \p_{y_j} \beta_2 (y) = \de_{1j} I_2 (y) + \int_{L_s}^{y_1} \p_{y_2} I_2 \p_{y_j} \bar{y}_2 (\tau; y) + \p_{y_3} I_2 \p_{y_j} \bar{y}_3 (\tau; y) d \tau , \, j = 1, 2, 3, \\
&
- \p_{y_i y_j}^2 \beta_2 (y) = \de_{1j} \p_{y_i} I_2 (y) + \int_{L_s}^{y_1}
\p_{y_2}^2 I_2 \p_{y_i} \bar{y}_2 (\tau; y) \p_{y_j} \bar{y}_2 (\tau; y)
+ \p_{y_3}^2 I_2 \p_{y_i} \bar{y}_3 (\tau; y) \p_{y_j} \bar{y}_3 (\tau; y) \\
& \quad\quad
+ \p_{y_2 y_3}^2 I_2 ( \p_{y_j} \bar{y}_2 (\tau; y) \p_{y_i} \bar{y}_3 (\tau; y) + \p_{y_i } \bar{y}_2 (\tau; y) \p_{y_j} \bar{y}_3 (\tau; y) )
+ \p_{y_2} I_2 \p_{y_i y_j}^2 \bar{y}_2 (\tau; y)
+ \p_{y_3} I_2 \p_{y_i y_j}^2 \bar{y}_3 (\tau; y) d \tau .
\end{aligned}\end{equation*}
Hence, there holds
\begin{equation}\no
\sum_{i = 2}^3 \| \beta_i (y) - y_i \|_{C^{2, \al } (\overline{\mb{D}})} \le C_* \| ({\bf \hat{v}}, \hat{v}_5) \|_{\mc{V}}.
\end{equation}

Differentiating $\eqref{ode1}_1$ with respect to $y_3$ and restricting the resulting equation on $y_3 = 0$ give
\begin{equation*}
\begin{cases}
  \f {d}{d \tau} \p_{y_3} \bar{y}_2 (\tau; y_1, y_2, \pm 1) = \p_{y_2} K_2 (\tau, \bar{y}_2 (\tau; y), \bar{y}_3 (\tau; y)) \p_{y_3} \bar{y}_2 (\tau; y_1, y_2, \pm 1), & \\
  \p_{y_3} \bar{y}_2 (y_1; y_1, y_2, \pm 1) = 0.
\end{cases}
\end{equation*}
Then $\p_{y_3} \bar{y}_2 (\tau; y_1, y_2, \pm 1) = 0$ for any $\tau \in [L_s, L_1]$. Similarly, one has $ \p_{y_2} \bar{y}_3 (\tau; y_1, \pm 1, y_3) = 0$ for any $\tau \in [L_s, L_1]$. Therefore,
\begin{equation}\label{bc24}
\begin{cases}
\p_{y_3} \beta_2 (y_1, y_2, \pm 1) = 0 , & \mbox{on } \Sigma_3^{\pm}, \\
  \p_{y_2} \beta_3 (y_1, \pm 1, y_3) = 0 , & \mbox{on } \Sigma_2^{\pm}.
\end{cases}
\end{equation}

By the characteristic method and \eqref{pf1}, one has
\begin{equation}\label{pf23}
\begin{aligned}
v_4 (y) & =  v_4 (L_s, \beta_2 (y), \beta_3 (y))\\
& =  b_2 v_5 (\beta_2 (y), \beta_3 (y)) + R_2 ({\bf \hat{v}} (L_s, \beta_2 (y), \beta_3 (y)), \hat{v}_5 (L_s, \beta_2 (y), \beta_3 (y))) \\
& = b_2 v_5 (y') + b_2 (v_5 (\beta_2 (y), \beta_3 (y)) - v_5 (y') ) + R_2 ({\bf \hat{v}} (L_s, \beta_2 (y), \beta_3 (y)), \hat{v}_5 (L_s, \beta_2 (y), \beta_3 (y))) \\
& = \f {b_2}{b_1} v_1 (L_s, y') + b_2 (v_5 (\beta_2 (y), \beta_3 (y)) - v_5 (y') ) + R_3 ({\bf \hat{v}} (L_s, \beta_2 (y), \beta_3 (y)), \hat{v}_5 (\beta_2 (y), \beta_3 (y)) ),
\end{aligned}
\end{equation}
where
\begin{equation*}
R_3 ({\bf \hat{v}} (L_s, y'), \hat{v}_5 ) = - \f {b_2}{b_1} R_1 ({\bf \hat{v}} (L_s, y'), \hat{v}_5 ) + R_2 ({\bf \hat{v}} (L_s, y'), \hat{v}_5 ).
\end{equation*}
Since $v_5 (y')$ is still unknown, \eqref{pf23} can be rewritten as
\begin{equation}\label{pf24}
v_4 (y_1, y') = \f {b_2}{b_1} v_1 (L_s, y') + R_4 ({\bf \hat{v}} (L_s, \beta_2 (y), \beta_3 (y)), \hat{v}_5 (\beta_2 (y), \beta_3 (y))),
\end{equation}
with
\begin{equation*}
R_4  = b_2 (\hat{v}_5 (\beta_2 (y), \beta_3 (y)) - \hat{v}_5 (y') ) + R_3 ({\bf \hat{v}} (L_s, \beta_2 (y), \beta_3 (y)), \hat{v}_5 (\beta_2 (y), \beta_3 (y)) ).
\end{equation*}
Therefore,
\begin{equation}\label{est4}
\begin{aligned}
& \| v_4 \|_{C^{2, \al} (\overline{\mb{D}})}
\le C_* ( \| v_1 (L_s, \cdot ) \|_{C^{2, \al } (\overline{E})} + \| R_4 \|_{C^{2, \al} (\overline{\mb{D}})} ) \\
&\le C_* ( \| v_1 (L_s, \cdot ) \|_{C^{2, \al } (\overline{E})}
+ \| \hat{v}_5 \|_{C^{3, \al } (\overline{E})} \sum_{i = 2}^3 \| \beta_i (y) - y_i \|_{C^{2, \al } (\overline{E})}) + C_* ( \epsilon \| ({\bf \hat{v}}, \hat{v}_5) \|_{\mc{V}} + \| ({\bf \hat{v}}, \hat{v}_5) \|_{\mc{V}}^2 ) \\
& \le C_* \| v_1 (L_s, \cdot ) \|_{C^{2, \al } (\overline{E})}
+ C_* (\eps \de_0 + \de_0^2).
\end{aligned}
\end{equation}

Since $({\bf \hat{v}}, \hat{v}_5) \in \mc{V}$ satisfies \eqref{bc20} and the incoming supersonic flow satisfies \eqref{bc14}, there holds
\begin{equation}\label{bc25}
\begin{cases}
   J_2 ({\bf \hat{v}} (L_s, y'), \hat{v}_5) |_{y_2 = \pm 1} = \p_{y_2}^2 J_2 ({\bf \hat{v}} (L_s, y'), \hat{v}_5) |_{y_2 = \pm 1} =0, & \forall y_3 \in [-1 , 1],  \\
   \p_{y_2} J_3 ({\bf \hat{v}} (L_s, y'), \hat{v}_5) |_{y_2 = \pm 1} = \p_{y_2} J ({\bf \hat{v}} (L_s, y'), \hat{v}_5) |_{y_2 = \pm 1}  =0,& \forall y_3 \in [-1 , 1], \\
   J_3 ({\bf \hat{v}} (L_s, y'), \hat{v}_5) |_{y_3 = \pm 1} = \p_{y_3}^2 J_3 ({\bf \hat{v}} (L_s, y'), \hat{v}_5) |_{y_3 = \pm 1}  =0, & \forall y_2 \in [-1 , 1],  \\
   \p_{y_3} J_2 ({\bf \hat{v}} (L_s, y'), \hat{v}_5) |_{y_3 = \pm 1} = \p_{y_3} J ({\bf \hat{v}} (L_s, y'), \hat{v}_5) |_{y_3 = \pm 1} = 0  & \forall y_2 \in [-1 , 1],
\end{cases}
\end{equation}
and for all $j = 1, 2$,
\begin{equation}\label{bc26}\begin{cases}
 \p_{y_2} \{R_{0j} ({\bf \hat{v}} (L_s, y'), \hat{v}_5)\} |_{y_2 = \pm 1} = 0 , & \forall y_3 \in [-1 , 1], \\
 \p_{y_3} \{R_{0j} ({\bf \hat{v}} (L_s, y'), \hat{v}_5)\} |_{y_3 = \pm 1} = 0 , & \forall y_2 \in [-1 , 1].
\end{cases}
\end{equation}
Thus, for $i = 1, 2$,
\begin{equation}\label{bc27}
\begin{cases}
\p_{y_2}\{ R_i ({\bf \hat{v}} (L_s, y'), \hat{v}_5)\} |_{y_2 = \pm 1} = 0 , & \forall y_3 \in [-1, 1], \\
\p_{y_3} \{R_i ({\bf \hat{v}} (L_s, y'), \hat{v}_5)\} |_{y_3 = \pm 1} = 0 , &\forall y_2 \in [-1, 1].
\end{cases}
\end{equation}
Together with \eqref{bc23} and \eqref{bc24}, one can conclude that
\begin{equation}\label{bc28}
\begin{cases}
   \p_{y_2}\{ R_4 ({\bf \hat{v}} (L_s, \beta_2 (y), \beta_3 (y)), \hat{v}_5 (\beta_2 (y), \beta_3 (y)))\} (y_1, \pm 1, y_3) = 0 , & \mbox{on } \Sigma_2^{\pm}, \\
  \p_{y_3} \{R_4 ({\bf \hat{v}} (L_s, \beta_2 (y), \beta_3 (y)), \hat{v}_5 (\beta_2 (y), \beta_3 (y)))\} (y_1, y_2, \pm 1) = 0 , & \mbox{on } \Sigma_3^{\pm},
\end{cases}
\end{equation}
and
\begin{equation}\no
\begin{cases}
  \p_{y_2} v_4 (y_1, \pm 1 , y_3) = \f {b_2}{b_1} \p_{y_2} v_1 (L_s, \pm 1, y_3), & \mbox{on } \Sigma_2^{\pm},  \\
  \p_{y_3} v_4 (y_1, y_2 , \pm 1) = \f {b_2}{b_1} \p_{y_2} v_1 (L_s, y_2 , \pm 1), & \mbox{on } \Sigma_3^{\pm}.
\end{cases}
\end{equation}

\textbf{Step 3.} Solving the transport equation of the first component of the vorticity. Thanks to \eqref{om7} and \eqref{bc6}, we just need to consider the following equation
\begin{equation}\label{pf5}
\begin{cases}
  \b( D_1^{\hat{v}_5} + \f {\hat{v}_2}{\bar{u} (D_0^{\hat{v}_5}) + \hat{v}_1} D_2^{\hat{v}_5} + \f {\hat{v}_3}{\bar{u}(D_0^{\hat{v}_5}) + \hat{v}_1} D_3^{\hat{v}_5} \b) \tilde{\om}_1  + \mu ({\bf \hat{v}} , \hat{v}_5) \tilde{\om}_1
 = H_0 ({\bf \hat{v}}, \hat{v}_5), &  \\
  \tilde{\om}_1 (L_s, y') = R_6 ({\bf \hat{v}} (L_s, y'), \hat{v}_5 (y')) , &
\end{cases}
\end{equation}
where
\begin{equation*}\begin{aligned}
& \mu ({\bf \hat{v}} (y), \hat{v}_5 (y'))= D^{\hat{v}_5}_2 \b( \f {\hat{v}_2}{\bar{u} (D_0^{\hat{v}_5}) + \hat{v}_1}\b) + D^{\hat{v}_5}_3 \b( \f {\hat{v}_3}{\bar{u} (D_0^{\hat{v}_5}) + \hat{v}_1} \b), \\
& R_6 ({\bf \hat{v}} (L_s, y'), \hat{v}_5 (y')) = \f 1{a_0} \b( \p_{y_2} g_3 ({\bf \hat{v}} (L_s, y'), \hat{v}_5(y')) - \p_{y_3} g_2 ({\bf \hat{v}} (L_s, y'), \hat{v}_5(y')) \b) + g_4 ({\bf \hat{v}} (L_s, y'), \hat{v}_5(y')).
\end{aligned}\end{equation*}
Since $({\bf \hat{v}}, \hat{v}_5) \in \mc{V}$ satisfies \eqref{bc20}, combining with \eqref{rh11}, \eqref{rh12} and \eqref{bc7}, there holds
\begin{equation}\label{bc30}
\begin{cases}
  \tilde{\om}_1 (L_s, \pm 1, y_3) = 0, & \forall y_3 \in [-1 ,1], \\
  \tilde{\om}_1 (L_s, y_2, \pm 1) = 0, & \forall y_2 \in [-1 , 1], \\
  H_0 ({\bf \hat{v}}, \hat{v}_5) (y_1, \pm 1, y_3) = 0, & \mbox{on } \Sigma_2^{\pm}, \\
  H_0 ({\bf \hat{v}}, \hat{v}_5) (y_1, y_2, \pm 1) = 0, & \mbox{on } \Sigma_3^{\pm}.
\end{cases}
\end{equation}

Integrating the equation \eqref{pf5} along the trajectory $(\tau, \bar{y}_2 (\tau;y), \bar{y}_3 (\tau;y))$ yields
\begin{equation}\label{om10}
\begin{aligned}
 \tilde{\om}_1 (y) &= R_6 (\beta_2 (y), \beta_3 (y)) e^{- \int_{L_s}^{y_1} \mu ({\bf \hat{v}}, \hat{v}_5) (t; \bar{y}_2 (t;y), \bar{y}_3 (t;y)) dt} \\
& \q +
\int_{L_s}^{y_1} H_0 ({\bf \hat{v}}, \hat{v}_5) (\tau ; \bar{y}_2 (\tau;y), \bar{y}_3 (\tau;y)) e^{- \int_{\tau }^{y_1} \mu ({\bf \hat{v}}, \hat{v}_5) (t; \bar{y}_2 (t;y), \bar{y}_3 (t;y)) dt } d\tau.
\end{aligned}
\end{equation}
Thus,
\begin{equation}\no
\begin{aligned}
\| \tilde{\om}_1 \|_{C^{1, \al} (\overline{\mb{D}})} &\le C_* (\| \tilde{\om}_1 (L_s, \cdot ) \|_{C^{1, \al} (\overline{\mb{D}})} + \| H_0 ({\bf \hat{v}}, \hat{v}_5) \|_{C^{1, \al} (\overline{\mb{D}})} ) \\
& \le
C_* (\eps \| ({\bf \hat{v}}, \hat{v}_5) \|_{\mc{V}} + \| ({\bf \hat{v}}, \hat{v}_5) \|_{\mc{V}}^2 )
\le C_* (\eps \de_0 + \de_0^2).
\end{aligned}
\end{equation}
Moreover, using \eqref{bc23}, \eqref{bc24}, \eqref{bc30} and \eqref{om10},
 the following compatibility conditions hold
\begin{equation}\label{bc31}
\begin{cases}
 \tilde{\om}_1 (y_1, \pm 1, y_3) = 0, & \mbox{on } \Sigma_2^{\pm}, \\
  \tilde{\om}_1 (y_1, y_2, \pm 1) = 0, & \mbox{on } \Sigma_3^{\pm}.
 \end{cases}
\end{equation}
Substituting \eqref{pf24} and \eqref{om10} into \eqref{euler9}-\eqref{euler11} yields
\be\label{euler13}
&& \p_{y_2} v_3 - \p_{y_3} v_2 = G_1 ({\bf \hat{v}}, \hat{v}_5), \\\label{euler14}
&& \p_{y_3} v_1 - \p_{y_1} v_3 - d_4 (y_1) \p_{y_3} v_1 (L_s, y') = G_2 ({\bf \hat{v}}, \hat{v}_5), \\\label{euler15}
&& \p_{y_1} v_2 - \p_{y_2} v_1 + d_4 (y_1) \p_{y_2} v_1 (L_s, y') = G_3 ({\bf \hat{v}}, \hat{v}_5),
\ee
where
\begin{equation*}\begin{aligned}
& d_4 (y_1) = \f {b_2}{b_1 \bar{u} (y_1)} , \\
& G_1 ({\bf \hat{v}}, \hat{v}_5) = \tilde{\om}_1 + H_1 ({\bf \hat{v}}, \hat{v}_5), \\
&G_2 ({\bf \hat{v}}, \hat{v}_5) = \f {\hat{v}_2 \tilde{\om }_1}{\bar{u} (D_0^{\hat{v}_5}) + \hat{v}_1} + H_2 ({\bf \hat{v}}, \hat{v}_5) + \f 1{\bar{u} (y_1)} \p_{y_3} R_4 ({\bf \hat{v}}, \hat{v}_5), \\
& G_3 ({\bf \hat{v}}, \hat{v}_5) = \f {\hat{v}_3 \tilde{\om }_1}{\bar{u} (D_0^{\hat{v}_5}) + \hat{v}_1} + H_3 ({\bf \hat{v}}, \hat{v}_5) - \f 1{\bar{u} (y_1)} \p_{y_2} R_4 ({\bf \hat{v}}, \hat{v}_5).
\end{aligned}\end{equation*}
The equations \eqref{bc20}, \eqref{bc28} and \eqref{bc31} yield that
\begin{equation}\label{bc32}
\begin{cases}
   G_1 ({\bf \hat{v}}, \hat{v}_5)|_{y_2 = \pm 1} = G_3 ({\bf \hat{v}}, \hat{v}_5)|_{y_2 = \pm 1} =\p_{y_2} G_2 ({\bf \hat{v}}, \hat{v}_5)|_{y_2 = \pm 1} =0, & \mbox{on } \Sigma_2^{\pm},  \\
   G_1 ({\bf \hat{v}}, \hat{v}_5)|_{y_3 = \pm 1} = G_2 ({\bf \hat{v}}, \hat{v}_5)|_{y_3 = \pm 1} = \p_{y_3} G_3 ({\bf \hat{v}}, \hat{v}_5)|_{y_3 = \pm 1} = 0,& \mbox{on } \Sigma_3^{\pm}.
\end{cases}
\end{equation}

Furthermore, \eqref{euler12} implies that
\begin{equation}\label{euler16}
d_1 (y_1) \p_{y_1} v_1 + \p_{y_2} v_2 + \p_{y_3} v_3 + d_2 (y_1) v_1+ \f {b_2}{b_1} d_3 (y_1) v_1 (L_s, y') = G_0 ({\bf \hat{v}}, {\hat{v}_5}),
\end{equation}
where
\begin{equation*}
 G_0 ({\bf \hat{v}}, {\hat{v}_5})= \mc{G}_0 ({\bf \hat{v}}, {\hat{v}_5})-  \f {b_2}{b_1} R_4 ({\bf \hat{v}} (L_s, \beta_2 (y), \beta_3 (y)), \hat{v}_5 (\beta_2 (y), \beta_3 (y))).
\end{equation*}
It follows from \eqref{bc9} and \eqref{pf24} that
\begin{equation}\label{bc12}
v_1 (L_1, y') - d_4 (L_1) v_1 (L_s, y') = q_4 (y'),
\end{equation}
with
\begin{equation*}\begin{aligned}
q_4 (y') =&  \f {b_1}{b_2} d_4 (L_1) R_4 ({\bf \hat{v}} (L_s, \beta_2 (y), \beta_3 (y)), \hat{v}_5 (\beta_2 (y), \beta_3 (y))) - \f {\eps P_{ex} (y')}{(\bar{\rho} \bar{u})(L_1)} \\
&
- \f 1{2 \bar{u} (L_1)} \sum_{i = 1}^3 v_i^2 (L_1, y') - \f 1{\bar{u} (L_1)} E ({\bf w} (L_1, y')).
\end{aligned}\end{equation*}
Using \eqref{1-15} and \eqref{bc10}, one can further verify that
\begin{equation}\no
\begin{cases}
  \p_{y_2} q_4 (\pm 1, y_3) = 0, & \forall y_3 \in [-1 , 1], \\
 \p_{y_3} q_4 (y_2, \pm 1) = 0, & \forall y_2 \in [-1 ,1].
\end{cases}
\end{equation}

\textbf{Step 4.} We have established a deformation-curl system for the velocity field, which is composed of equations \eqref{euler13}-\eqref{euler15}, \eqref{euler16} with the boundary consitions \eqref{rh20}-\eqref{bc17}, \eqref{bc8} and \eqref{bc12}, where $q_1$ and $q_i^{\pm} (i = 2, 3)$ are evaluated at $({\bf \hat{v}}, \hat{v}_5)$. However, due to the linearization, the vector field $(G_1, G_2, G_3) ({\bf \hat{v}}, \hat{v}_5)$ may not be divergence-free and hence the solvability condition of the curl system \eqref{euler13}-\eqref{euler15} may not hold in general. To overcome this difficulty, we first consider the following enlarged deformation-curl system, which involves an additional new unknown function $\Pi$ with homogeneous Dirichlet boundary condition for $\Pi$.
\begin{equation}\label{pf6}
\begin{cases}
  d_1 (y_1) \p_{y_1} v_1 + \p_{y_2} v_2 + \p_{y_3} v_3 + d_2 (y_1) v_1+ \f {b_2}{b_1} d_3 (y_1) v_1 (L_s, y') = G_0 ({\bf \hat{v}}, {\hat{v}_5}), & \text{in } \mb{D}, \\
  \p_{y_2} v_3 - \p_{y_3} v_2 + \p_{y_1} \Pi = G_1 ({\bf \hat{v}}, \hat{v}_5), & \text{in } \mb{D}, \\
  \p_{y_3} v_1 - \p_{y_1} v_3 - d_4 (y_1) \p_{y_3} v_1 (L_s, y') + \p_{y_2} \Pi = G_2 ({\bf \hat{v}}, \hat{v}_5), & \text{in } \mb{D}, \\
  \p_{y_1} v_2 - \p_{y_2} v_1 + d_4 (y_1) \p_{y_2} v_1 (L_s, y') + \p_{y_3} \Pi = G_3 ({\bf \hat{v}}, \hat{v}_5), & \text{in } \mb{D}, \\
  (\p_{y_2}^2 + \p_{y_3}^2) v_1 (L_s, y') = a_0 b_1 \p_{y_2} v_2 (L_s, y') + a_0 b_1 \p_{y_3} v_3 (L_s, y') + q_1 ({\bf \hat{v}} (L_s, y'), \hat{v}_5(y')), & \forall y' \in E, \\
  v_2 (y_1, \pm 1 , y_3) = \Pi (y_1, \pm 1, y_3) = 0, & \text{on } \Sigma_2^{\pm}, \\
  v_3 (y_1, y_2, \pm 1) = \Pi (y_1, y_2, \pm 1) = 0, & \text{on } \Sigma_3^{\pm}, \\
  v_1 (L_1, y') - d_4 (L_1) v_1 (L_s, y') = q_4 (y'), & \forall y' \in E, \\
  \Pi (L_s, y') = \Pi (L_1, y') = 0, & \forall y' \in E.
\end{cases}
\end{equation}
The system \eqref{pf6} should be supplemented with the boundary conditions \eqref{bc4}-\eqref{bc17} where $q_i^{\pm} (i = 2, 3)$ are evaluated at $({\bf \hat{v}}, \hat{v}_5)$ such that there is a unique solvability result. On the intersection of the shock front with the nozzle wall, there holds
\be\label{bc34}
(\p_{y_2} v_1 - a_0 b_1 v_2) (L_s, \pm 1, y_3) = 0, \, \forall y_3 \in [-1, 1], \\\label{bc35}
(\p_{y_3} v_1 - a_0 b_1 v_3) (L_s, y_2, \pm 1) = 0, \, \forall y_2 \in [-1, 1],
\ee
which follow from \eqref{rh12}-\eqref{rh13}, \eqref{bc4}-\eqref{bc17} \eqref{bc20} \eqref{bc25}, and \eqref{bc27}.

We will apply Duhamel's principle to show the unique solvability of the problem \eqref{pf6} with \eqref{bc34}-\eqref{bc35} in the following steps.

\textbf{Step 4.1 } First, by taking the divergence operator to the second, third and fourth equations in \eqref{pf6}, one gets
\begin{equation}\label{pf7}
\begin{cases}
  \p_{y_1}^2 \Pi + \p_{y_2}^2 \Pi + \p_{y_3}^2 \Pi = \p_{y_1} G_1 + \p_{y_2} G_2 + \p_{y_3} G_3, & \text{in}\ \ \mb{D},  \\
  \Pi (L_s, y') = \Pi (L_1, y') = 0, & \forall y' \in E, \\
  \Pi (y_1, \pm 1, y_3) = 0, & \text{on } \Sigma_2^{\pm}, \\
  \Pi (y_1, y_2, \pm 1) = 0, & \text{on } \Sigma_3^{\pm}.
\end{cases}\end{equation}
\par By the standard elliptic theory in \cite{GT82}, the system \eqref{pf7} has a unique solution $\Pi\in C^{2, \al} (\mb{D}) \cap C^{1, \al } (\overline{\mb{D}})$. To deal with the regularity near the corner, we use the the standard symmetric extension technique to extend the functions $\Pi$, $G_1$, $G_2$,$G_3$ as follows:
\begin{equation}\no
(\tilde{\Pi }, \tilde{G}_1, \tilde{G}_2) (y) =
\begin{cases}
  - (\Pi, G_1, G_2) (y_1, y_2, 2 - y_3), & y \in [L_s, L_1] \times [-1, 1] \times [1, 3],  \\
  (\Pi, G_1, G_2) (y_1, y_2, y_3), & y \in [L_s, L_1] \times [-1 , 1] \times [-1 , 1], \\
  - (\Pi, G_1, G_2) (y_1, y_2, -2 - y_3 ), & y \in [L_s, L_1] \times [-1 , 1] \times [-3, -1 ],
\end{cases}
\end{equation}
and
\begin{equation}\label{ext2}
\tilde{G}_3 (y) =
\begin{cases}
   G_3 (y_1, y_2, 2 - y_3), & y \in [L_s, L_1] \times [-1, 1] \times [1, 3],  \\
   G_3 (y_1, y_2, y_3), & y \in [L_s, L_1] \times [-1, 1] \times [-1, 1], \\
   G_3 (y_1, y_2, -2 - y_3 ), & y \in [L_s, L_1] \times [-1, 1] \times [-3, -1].
\end{cases}
\end{equation}
The extension of $\Pi $, $G_i$, $i = 1, 2, 3$ along the $x_2$ direction can be done similarly.

\par Then it follows from \eqref{bc32} that
\begin{equation}\no
\begin{cases}
    \p_{y_1}^2 \tilde{\Pi} + \p_{y_2}^2 \tilde{\Pi} + \p_{y_3}^2 \tilde{\Pi} = \p_{y_1} \tilde{G}_1 + \p_{y_2} \tilde{G}_2 + \p_{y_3} \tilde{G}_3 \in C^{1, \al} (\mb{D}_e), & \\
    \tilde{\Pi} (L_s, y_2, y_3) = \tilde{\Pi} (L_1, y_2, y_3) =0, \, \forall (y_2, y_3) \in [-3, 3] \times [-3, 3],
\end{cases}\end{equation}
where $\mb{D}_e = (L_s, L_1) \times (-3, 3) \times (-3, 3)$. Therefore, one can improve the regularity of $\Pi$ on $\mb{D}$ to be $C^{2, \al} (\overline{\mb{D}})$ with the estimate
\begin{equation}\no
\| \Pi \|_{C^{2, \al} (\overline{\mb{D}})} \le C_* \sum_{i = 1}^3 \| G_i \|_{C^{1, \al} (\overline{\mb{D}})}
\le C_* (\eps \| ({\bf \hat{v}}, \hat{v}_5) \|_{\mc{V}} + \| ({\bf \hat{v}}, \hat{v}_5) \|_{\mc{V}}^2 )
\le C_* (\eps \de_0 + \de_0^2).
\end{equation}
In addition,  the following compatibility conditions hold
\begin{equation}\label{bc36}
\begin{cases}
   \p_{y_1} \Pi (y_1, \pm 1, y_3)  = \p_{y_3} \Pi (y_1, \pm 1, y_3) = 0, & \mbox{on } \Sigma_2^{\pm}, \\
   \p_{y_1} \Pi (y_1, y_2, \pm 1)  = \p_{y_2} \Pi (y_1, y_2, \pm 1) = 0, & \mbox{on } \Sigma_3^{\pm}.
\end{cases}
\end{equation}

\textbf{Step 4.2 } Next we solve the divergence-curl system with normal boundary conditions
\begin{equation}\label{pf8}
\begin{cases}
  \p_{y_1} \tilde{v}_1 + \p_{y_2} \tilde{v}_2 + \p_{y_3} \tilde{v}_3 = 0, & \text{in } \mb{D}, \\
  \p_{y_2} \tilde{v}_3 - \p_{y_3} \tilde{v}_2 = G_1 ({\bf \hat{v}}, \hat{v}_5) - \p_{y_1} \Pi \co \tilde{G}_1, & \text{in } \mb{D}, \\
  \p_{y_3} \tilde{V}_1 - \p_{y_1} \tilde{v}_3 = G_2 ({\bf \hat{v}}, \hat{v}_5) - \p_{y_2} \Pi \co \tilde{G}_2, & \text{in } \mb{D}, \\
  \p_{y_1} \tilde{V}_2 - \p_{y_2} \tilde{v}_1 = G_3 ({\bf \hat{v}}, \hat{v}_5) - \p_{y_3} \Pi \co \tilde{G}_3, & \text{in } \mb{D}, \\
  \tilde{v}_1 (L_s, y_2, y_3) = \tilde{v}_1 (L_1, y_2, y_3) = 0, & \forall y' \in E, \\
   \tilde{v}_2 (y_1, \pm 1, y_3) = 0, & \text{on } \Sigma_2^{\pm}, \\
   \tilde{v}_3 (y_1, y_2, \pm 1) = 0, & \text{on } \Sigma_3^{\pm}.
\end{cases}
\end{equation}

According to \eqref{pf7}, one has
\begin{equation}\no
\p_{y_1} \tilde{G}_1 + \p_{y_2} \tilde{G}_2 + \p_{y_3} \tilde{G}_3 \equiv 0 , \q \text{in } \mb{D}.
\end{equation}
Also it follows from \eqref{bc32} and \eqref{bc36} that
\begin{equation}\label{bc37}
\begin{cases}
  \tilde{G}_1 (y_1, \pm 1, y_3) = \tilde{G}_3 (y_1, \pm 1, y_3) = \p_{y_2} \tilde{G}_2 (y_1, \pm 1 , y_3) = 0, & \mbox{on } \Sigma_2^{\pm}, \\
  \tilde{G}_1 (y_1, y_2, \pm 1) = \tilde{G}_2 (y_1, y_2, \pm 1)= \p_{y_3} \tilde{G}_3 (y_1, y_2, \pm 1) =0,& \mbox{on } \Sigma_3^{\pm}.
\end{cases}
\end{equation}

By the theory in \cite{KY09}, the divergence-curl system with the homogeneous boundary conditions is uniquely solvable. Thanks to the compatibility condition \eqref{bc37} and the symmetric extension technique as above, there exists a unique $C^{2, \al } (\overline{\mb{D}})$ solution to \eqref{pf8} with
\begin{equation}\no
\begin{aligned}
& \sum_{i = 1}^3 \| \tilde{v}_i \|_{C^{2, \al } (\overline{\mb{D}})}
\le C_* \sum_{i = 1}^3 \| \tilde{G}_i \|_{C^{1, \al } (\overline{\mb{D}})}
\le C_* \b( \sum_{i = 1}^3 \| G_i \|_{C^{1, \al } (\overline{\mb{D}})} + \| \Pi \|_{C^{2, \al } (\overline{\mb{D}})} \b)
 \\
& \le C_* \sum_{i = 1}^3 \| G_i \|_{C^{1, \al } (\overline{\mb{D}})} \le C_* (\eps \| ({\bf \hat{v}}, \hat{v}_5) \|_{\mc{V}} + \| ({\bf \hat{v}}, \hat{v}_5) \|_{\mc{V}}^2 )
\le C_* (\eps \de_0 + \de_0^2),
\end{aligned}\end{equation}
and the following compatibility conditions
\begin{equation}\label{bc38}
\begin{cases}
   (\tilde{v}_2, \p_{y_2}^2 \tilde{v}_2) (y_1, \pm 1, y_3) = \p_{y_2} (\tilde{v}_1, \tilde{v}_3) (y_1, \pm 1, y_3)  =0, & \mbox{on } \Sigma_2^{\pm}, \\
   (\tilde{v}_3, \p_{y_3}^2 \tilde{v}_3) (y_1, y_2, \pm 1) = \p_{y_3} (\tilde{v}_1, \tilde{v}_2) (y_1, y_2, \pm 1) =0,& \mbox{on } \Sigma_3^{\pm}.
\end{cases}
\end{equation}

\textbf{Step 4.3 } Let $(v_1, v_2, v_3)$ be the solution to \eqref{pf6}, and define
\begin{equation*}
N_j (y) = v_j (y) - \tilde{v}_j (y), \q j = 1, 2, 3.
\end{equation*}
Then $N_j$, $j = 1, 2, 3$ solve the following system
\begin{equation}\label{pf9}
\begin{cases}
  d_1 (y_1) \p_{y_1} N_1 + \p_{y_2} N_2 + \p_{y_3} N_3 + d_2 (y_1) N_1+\f {b_2}{b_1} d_3 (y_1) N_1 (L_s, y')= G_4 ({\bf \hat{v}}, \hat{v}_5), &  \text{in } \mb{D}, \\
  \p_{y_2} N_3 - \p_{y_3} N_2 = 0, & \text{in } \mb{D}, \\
  \p_{y_3} (N_1 - d_4 (y_1) N_1 (L_s, y') ) - \p_{y_1} N_3 = 0, & \text{in } \mb{D}, \\
  \p_{y_1} N_2 - \p_{y_1} (N_1 -d_4 (y_1) N_1 (L_s, y')) = 0, & \text{in } \mb{D}, \\
  (\p_{y_2}^2 + \p_{y_3}^2) N_1 (L_s, y') - a_0 b_1 (\p_{y_2} + \p_{y_3}) N_1 (L_s, y') = q_5 ({\bf \hat{v}} (L_s, y'), \hat{v}_5 (y')), & \forall y' \in E, \\
  N_2 (y_1, \pm 1, y_3) = 0, & \text{on } \Sigma_2^{\pm}, \\
  N_3 (y_1, y_2, \pm 1) = 0, & \text{on } \Sigma_3^{\pm}, \\
  N_1 (L_1, y') - d_4 (L_1) N_1 (L_1, y') = q_4 ({\bf \hat{v}} (L_s, y'), \hat{v}_5 (y')), & \forall y' \in E,
\end{cases}
\end{equation}
where
\begin{equation*}\begin{aligned}
& G_4 ({\bf \hat{v}}, \hat{v}_5) = G_0 ({\bf \hat{v}}, \hat{v}_5) + \bar{M}^2 (y_1) \p_{y_1} \tilde{v}_1 - d_2 (y_1) \tilde{v}_1-\f {b_2}{b_1} d_3 (y_1) \tilde{v}_1 (L_s, y'), \\
& q_5 ({\bf \hat{v}} (L_s, y'), \hat{v}_5 (y')) = q_1 ({\bf \hat{v}} (L_s, y'), \hat{v}_5 (y')) + a_0 b_1 (\p_{y_2} \tilde{v}_2 + \p_{y_3} \tilde{v}_3) (L_s, y').
\end{aligned}\end{equation*}
It follows from the boundary conditions \eqref{bc34}-\eqref{bc35} and the compatibility condition \eqref{bc38} that
\be\label{bc39}
&&  ( \p_{y_2} N_1 - a_0 b_1 N_2 ) (L_s, \pm 1, y_3) = 0 , \q \forall y_3 \in [-1 , 1], \\\label{bc40}
&& (\p_{y_3} N_1 - a_0 b_1 N_3) (L_s, y_2, \pm 1) = 0, \q \forall y_2 \in [-1, 1].
\ee

 \par By the second, third and fourth equations in \eqref{pf9}, there exists a potential function $\phi $ such that
\begin{equation*}
N_1 (y_1, y') - d_4 (y_1) N_1 (L_s, y') = \p_{y_1} \phi, \quad N_2 (y_1, y') = \p_{y_2} \phi, \quad N_3 (y_1, y') = \p_{y_3} \phi.
\end{equation*}
Therefore
\begin{equation*}
 N_1 (L_s, y') = \f 1{b_3} \p_{y_1} \phi (L_s, y'), \quad
 N_1 (y_1, y') = \p_{y_1} \phi (y_1, y') + \f {d_4 (y_1)}{b_3} \p_{y_1} \phi (L_s, y'), \quad
b_3 = 1 - d_4 (L_s) > 0.
\end{equation*}
Then it follows from \eqref{pf9} that $\phi$ satisfies the following equations
\begin{equation}\label{pf10}
\begin{cases}
  d_1 (y_1) \p_{y_1}^2 \phi + \p_{y_2}^2 \phi + \p_{y_3}^2 \phi + d_2 (y_1) \p_{y_1} \phi + \f 1 {b_3} d_5 (y_1) \p_{y_1} \phi (L_s, y' ) = G_4 ({\bf \hat{v}}, \hat{v}_5), & \text{in } \mb{D}, \\
  (\p_{y_2}^2 + \p_{y_3}^2 ) (\p_{y_1} \phi (L_s, y') - b_4 \phi (L_s, y')) = b_3 q_5 ({\bf \hat{v}} (L_s, y'), \hat{v}_5 (y')), & \forall y' \in E, \\
  \p_{y_2} \phi (y_1, \pm 1, y_3) = 0, & \text{on } \Sigma_2^{\pm}, \\
  \p_{y_3} \phi (y_1, y_2, \pm 1) = 0, & \text{on } \Sigma_3^{\pm}, \\
  \p_{y_1} \phi (L_1, y') = q_4 ({\bf \hat{v}} (L_s, y'), \hat{v}_5 (y')), & \forall y' \in E,
\end{cases}
\end{equation}
where
\begin{equation*}
\begin{aligned}
 d_5 (y_1)  &=  d_1 (y_1) d_4' (y_1) + d_2 (y_1) d_4 (y_1)+ \f {b_2 d_3 (y_1)}{b_1} \\
&=  - (1 - \bar{M}^2 (y_1)) \f {b_2 \bar{u}' (y_1)}{b_1 \bar{u}^2 (y_1)} + \f { \bar{f} (y_1)-( \ga+1) \bar{u}\bar u^\prime(y_1)}{c^2 (\bar{\rho})} \f {b_2}{b_1 \bar{u} (y_1)} +\f {(\gamma-1) b_2  \bar{u}' (y_1)}{b_1  c^2 (\bar{\rho}(y_1))}\\
& = \f {2b_2 \bar{f}}{b_1 \bar{u} (y_1) (c^2 (\bar{\rho}) - \bar{u}^2)}
 < 0, \\
 b_4 &= a_0 b_1 b_3 > 0 .
 \end{aligned}
\end{equation*}

We can rewrite the boundary conditions \eqref{bc39}-\eqref{bc40} as
\be\label{bc41}
&&  \p_{y_2} (\p_{y_1} \phi - b_4 \phi ) (L_s, \pm 1, y_3) = 0, \q \forall y_3 \in [-1 , 1], \\\label{bc42}
&& \p_{y_3} (\p_{y_1} \phi - b_4 \phi ) (L_s, y_2, \pm 1 ) = 0, \q \forall y_2 \in [-1, 1].
\ee

To solve the problem \eqref{pf10} with \eqref{bc41}-\eqref{bc42}, we need the following Lemma.
\begin{lemma}
  One the shock front $\{ (L_s, y') : y' \in E \}$, there exists a unique $C^{2, \al } (\overline{E})$ function $m_1 (y')$ such that
  \begin{equation}\no
  \p_{y_1} \phi (L_s, y') - b_4 \phi (L_s, y') = m_1 (y'),
  \end{equation}
  where $m_1 (y')$ satisfies the Poisson equation with the homogeneous Neumann boundary conditions
  \begin{equation}\no
  \begin{cases}
    (\p_{y_2}^2 + \p_{y_3}^2) m_1 (y') = b_3 q_5 ({\bf \hat{v}} (L_s, y'), \hat{v}_5 (y')) , & \mbox{in } E \co (-1, 1) \times (-1, 1), \\
    \p_{y_2} m_1 (\pm 1, y_3) = 0, & \forall y_3 \in [-1, 1], \\
    \p_{y_3} m_1 (y_2, \pm 1) = 0, & \forall y_2 \in [-1, 1],
  \end{cases}
  \end{equation}
  and the condition
  \begin{equation}\no
  \iint_D m_1 (y_2, y_3) d y' = 0.
  \end{equation}
\end{lemma}
 The proof of this lemma is similar to \cite[Lemma 3.2]{WX23}, here we omit the details.  Then it follows from \eqref{pf10}, \eqref{bc41}-\eqref{bc42} and the above Lemma that
\begin{equation}\label{pf13}
\begin{cases}
  \p_{y_1} (d_1 (y_1) \p_{y_1} \phi) + \p_{y_2}^2 \phi + \p_{y_3}^2 \phi + d_6 (y_1) \p_{y_1} \phi + a_0 b_1 d_5 (y_1) \phi (L_s, y') = G_5 (y), & \mbox{in } \mb{D}, \\
  \p_{y_1} \phi (L_s, y') - b_4 \phi (L_s, y') = m_1 (y'), & \forall y' \in E, \\
  \p_{y_2} \phi (y_1, \pm 1, y_3) = 0, & \text{on } \Sigma_2^{\pm}, \\
  \p_{y_3} \phi (y_1, y_2, \pm 1) = 0, & \text{on } \Sigma_3^{\pm}, \\
  \p_{y_1} \phi (L_1, y') = m_2 (y'), & \forall y' \in E,
\end{cases}
\end{equation}
where
\begin{equation*}\begin{aligned}
& G_5 (y) = G_4 (y) - \f {d_5 (y_1)}{b_3} m_1 (y'), \\
& d_6 (y_1) = - d_1' (y_1) + d_2 (y_1), \ \ m_2 (y') = q_4 ({\bf \hat{v}} (L_s, y'), \hat{v}_5 (y')).
\end{aligned}\end{equation*}
The simple calculations yield that
\begin{equation}\label{bc43}
\begin{cases}
  \p_{y_2} G_5 (y_1, \pm 1, y_3) = 0, & \mbox{on } \Sigma_2^{\pm}, \\
  \p_{y_3} G_5 (y_1, y_2, \pm 1) = 0, & \mbox{on } \Sigma_3^{\pm}, \\
  \p_{y_2} m_2 (\pm 1, y_3) = \p_{y_2} m_3 (\pm 1, y_3) = 0, & \forall y_3 \in [-1 , 1],  \\
  \p_{y_3} m_2 (y_2, \pm 1) = \p_{y_3} m_3 (y_2, \pm 1) = 0, & \forall y_2 \in [-1 , 1] .
\end{cases}
\end{equation}

The oblique boundary condition at $y_1 = L_s$ allows us to replace $\f {d_5 (y_1)}{b_3} \p_{y_1} \phi (L_s, y')$ in the first equation of \eqref{pf10} by $a_0 b_1 d_5 (y_1) \phi (L_s, y')$ in \eqref{pf13}. This simplifies the unique solvability of the problem \eqref{pf13}, as only the trace $\phi (L_s, y')$ not the derivative $\p_{y_1} \phi (L_s, y')$ appears in the nonlocal term. Thus, the existence and uniqueness of the solution to \eqref{pf13} can be shown by the Lax-Milgram theorem and the Fredholm alternatives.

In fact, we first obtain the existence of the weak solution to \eqref{pf13}. We call $\phi \in H^1 (\mb{D})$ is a weak solution to \eqref{pf13}, if for any $\psi \in H^1 (\mb{D})$, there exists
\begin{equation}\label{pf14}
\mc{B} (\phi, \psi) = \mc{L} (\psi), \ \ \forall \psi \in H^1 (\mb{D}),
\end{equation}
where
\begin{equation*}
\begin{aligned}
 \mc{B} (\phi, \psi ) &= \iiint_{\mb{D}} d_1 (y_1) \p_{y_1} \phi \p_{y_1} \psi + \p_{y_2} \phi \p_{y_2} \psi + \p_{y_3} \phi \p_{y_3} \psi - d_6 (y_1) \p_{y_1} \phi \psi \\
& \q \q - a_0 b_1 d_5 (y_1) \phi (L_s, y') \psi d y_1 dy'
+ \iint_E d_1 (L_s) b_4 \phi (L_s, y') \psi (L_s, y') d y', \\
 \mc{L} (\psi) &= - \iiint_{\mb{D}} G_5 (y) \psi (y) d y_1 dy' + \iint_E d_1 (L_1) m_2 (y') \psi (L_1, y') - d_1 (L_s) m_1 (y') \psi (L_1, y') d y'.
 \end{aligned}
\end{equation*}
We next solve \eqref{pf14}.
\begin{lemma}\label{lem2}
  There exists a positive constant $K$ depending only on the background solution such that the following problem has a unique weak solution in $H^1 (\mb{D})$
  \begin{equation}\label{pf29}
  \begin{cases}
  \p_{y_1} (d_1 (y_1) \p_{y_1} \phi) + \p_{y_2}^2 \phi + \p_{y_3}^2 \phi + d_6 (y_1) \p_{y_1} \phi + a_0 b_1 d_5 (y_1) \phi (L_s, y') - K \phi = G_5 (y), & \mbox{in } \mb{D}, \\
  \p_{y_1} \phi (L_s, y') - b_4 \phi (L_s, y') = m_1 (y'), & \forall y' \in E, \\
  \p_{y_2} \phi (y_1, \pm 1, y_3) = 0, & \text{on } \Sigma_2^{\pm}, \\
  \p_{y_3} \phi (y_1, y_2, \pm 1) = 0, & \text{on } \Sigma_3^{\pm}, \\
  \p_{y_1} \phi (L_1, y') = m_2 (y'), & \forall y' \in E.
\end{cases}
  \end{equation}
\end{lemma}
\begin{proof}
  We call a $H^1 (\mb{D})$ function $\phi$ a weak solution of the problem \eqref{pf29}, if for any $\psi \in H^1 (\mb{D})$, there holds
  \begin{equation}\label{pf30}
  \mc{B}_K (\phi, \psi ) \co \mc{B} (\phi, \psi ) + K \iiint_{\mb{D}} \phi \psi d y = \mc{L} (\psi ), \q \forall \psi \in H^1 (\mb{D}).
  \end{equation}

  For any $\eps > 0$,
  \begin{equation*}
  \iint_E \phi^2 (L_s, y') d y' \le \f {C_0}{\eps } \iiint_{\mb{D}} \phi^2 (y_1, y') d y' dy_1 + \eps \iiint_{\mb{D}} (\p_{y_1} \phi)^2 (y_1, y') d y' dy_1.
  \end{equation*}
  Then
  \begin{equation*}
  |\mc{B}_K (\phi, \psi)| \le C_0 \| \phi \|_{H^1 (\mb{D})} \| \psi \|_{H^1 (\mb{D})}, \quad
   |\mc{L} (\psi)| \le C_0 (\| G_5 \|_{L^2 (\mb{D})} + \sum_{i = 1}^2 \| m_i \|_{L^2 (E)}) \| \psi \|_{H^1 (\mb{D})},
  \end{equation*}
  and
  \begin{equation*}
  \begin{aligned}
   \mc{B}_K (\phi, \phi ) &= \iiint_{\mb{D}} d_1 (y_1) | \p_{y_1} \phi |^2 + | \p_{y_2} \phi |^2  + | \p_{y_3} \phi|^2 - d_6 (y_1) \p_{y_1} \phi \phi \\
  &\q - a_0 b_1 d_5 (y_1) \phi (L_s, y') \phi (y_1, y') d y_1 dy'
  + \iint_E d_1 (L_s) b_4 | \phi (L_s, y')|^2 d y'
  + K \iiint_{\mb{D}} |\phi|^2 d y \\\no
  &\ge
  C_* (\| \na \phi \|_{L^2 (\mb{D})}^2 + \| \phi (L_s, \cdot ) \|_{L^2 (E)}^2) + K \| \phi \|_{L^2 (\mb{D})}^2 - \f {C_*}{4} \| \p_{y_1} \phi \|_{L^2 (\mb{D})}^2 - \tilde{C}_* \| \phi \|_{L^2 (\mb{D})}^2 - \f {C_*}{4}  \| \phi (L_s, \cdot ) \|_{L^2 (E)}^2 \\
  & \ge
  \f {C_*}2 (\| \na \phi \|_{L^2 (\mb{D})}^2 + \| \phi (L_s, \cdot ) \|_{L^2 (E)}^2) + \f K2 \| \phi \|_{L^2 (\mb{D})}^2,
  \end{aligned}
  \end{equation*}
  provided that $K$ is sufficiently large. Then the existence and uniqueness of $H^1 (\mb{D})$ solution $\phi$ to \eqref{pf30} can be obtained by using the Lax-Milgram theorem. Therefore, we complete this proof.
\end{proof}

We are going to solve the problem \eqref{pf13}.
\begin{proposition}
  Suppose that $G_5 \in C^{1, \al} (\overline{\mb{D}})$ and $m_i \in C^{2, \al} (\overline{E})$, $i = 1, 2$ satisfy \eqref{bc43}. Then there exists a unique $C^{3, \al } (\overline{\mb{D}})$ solution to the problem \eqref{pf13} with the estimate
  \begin{equation}\label{est9}
  \| \phi \|_{C^{3, \al } (\overline{\mb{D}})} \le C_* (\| G_5 \|_{C^{1, \al} (\overline{\mb{D}})} + \sum_{i = 1}^2 \| m_i \|_{C^{2, \al} (\overline{E})}),
  \end{equation}
  where $C_*$ depends only on $d_1$, $d_5$, $d_6$, $b_3$, $b_4$ and thus depends only on the background solution.
\end{proposition}

\begin{proof}
  First, the regularity of $H^1 (\mb{D})$ weak solution to \eqref{pf13} can be improved to $C^{3, \al } (\overline{\mb{D}})$ with the following estimate
  \begin{equation}\label{est10}
  \| \phi \|_{C^{3, \al } (\overline{\mb{D}})} \le C_* (\| \phi \|_{H^1 (\overline{\mb{D}})} + \| G_5 \|_{C^{1, \al} (\overline{\mb{D}})} + \sum_{i = 1}^2 \| m_i \|_{C^{2, \al} (\overline{D})}).
  \end{equation}

  Rewrite the equation \eqref{pf13} as
  \begin{equation*}
  \begin{cases}
  \p_{y_1} (d_1 (y_1) \p_{y_1} \phi) + \p_{y_2}^2 \phi + \p_{y_3}^2 \phi + d_6 (y_1) \p_{y_1} \phi = G_6 (y) \co G_5 (y) - a_0 b_1 d_5 (y_1) \phi (L_s, y'), & \mbox{in } \mb{D}, \\
  \p_{y_1} \phi (L_s, y') - b_4 \phi (L_s, y') = m_1 (y'), & \forall y' \in E, \\
  \p_{y_2} \phi (y_1, \pm 1, y_3) = 0, & \text{on } \Sigma_2^{\pm}, \\
  \p_{y_3} \phi (y_1, y_2, \pm 1) = 0, & \text{on } \Sigma_3^{\pm}, \\
  \p_{y_1} \phi (L_1, y') = m_2 (y'), & \forall y' \in E.
\end{cases}
  \end{equation*}
   It follows from the trace theorem that $\phi (L_s, y') \in L^2 (E)$.
   By applying \cite[Theorems 5.36 and 5.45]{LG13}, the global $L^{\infty}$ bound and $C^{\al_1}$ estimates( for some $\al_1 \in (0, 1)$) on
 $\phi$ can be obtain as follows
  \begin{equation*}
  \begin{aligned}
   \| \phi \|_{C^{0, \al_1 } (\overline{E}) } & \le C_* (\| a_0 b_1 d_5 (y_1) \phi (L_s, \cdot) \|_{L^2 (E)} + \| G_5 \|_{L^4 (\mb{D})} + \sum_{i = 1}^2 \| m_i \|_{C^{1, \al} (\overline{E})} ) \\
  &\le C_*  (\| \phi \|_{H^1 (\mb{D})} + \| G_5 \|_{C^{1, \al} (\overline{\mb{D}})} + \sum_{i = 1}^2 \| m_i \|_{C^{1, \al} (\overline{E})} ).
  \end{aligned}
  \end{equation*}
  Thus the term $a_0 b_1 d_5 (y_1) \phi (L_s, \cdot) \in C^{\al_1} (\overline{\mb{D}})$, and then the Schauder estimate in \cite[Theorem 4.6]{LG13}
 implies that
  \begin{equation}\no
  \| \phi \|_{C^{1, \al_1 } (\overline{E}) } \le C_*  (\| \phi \|_{H^1 (\mb{D})} + \| G_5 \|_{C^{1, \al} (\overline{\mb{D}})} + \sum_{i = 1}^2 \| m_i \|_{C^{1, \al} (\overline{E})} ).
  \end{equation}
  Next, we extend the function $\phi$, $G_6$ to $\mb{D}_e$ as  in \eqref{ext2}, and extend $m_i$, $i = 1, 2$ as
  \begin{equation}\no
  \tilde{m}_i (y') = \begin{cases}
                       m_i (y_2, 2 - y_3), & y' \in [-1, 1] \times [1, 3], \\
                       m_i (y_2, y_3), & y' \in [-1, 1] \times [-1, 1],  \\
                       m_i (y_2, -2 - y_3), & y' \in [-1, 1] \times [-3, -1].
                     \end{cases}
  \end{equation}
  The extension guarantees that $\tilde{G}_6 \in C^{1, \al} (\overline{\mb{D}_e})$ and $\tilde{m}_i \in C^{2, \al } ([-3, 3] \times [-3, 3])$ by the compatibility conditions \eqref{bc43}. Therefore,
  \begin{equation}\no
  \begin{cases}
  \p_{y_1} (d_1 (y_1) \p_{y_1} \tilde{\phi}) + \p_{y_2}^2 \tilde{\phi} + \p_{y_3}^2 \tilde{\phi} + d_6 (y_1) \p_{y_1} \tilde{\phi} = \tilde{G}_6 (y) , & \mbox{in } \mb{D}_e, \\
  \p_{y_1} \tilde{\phi} (L_s, y') - b_4 \tilde{\phi} (L_s, y') = \tilde{m}_1 (y'), & \forall y' \in [-3, 3] \times [-3, 3], \\
  \p_{y_1} \tilde{\phi} (L_1, y') = \tilde{m}_2 (y'), & \forall y' \in [-3, 3] \times [-3, 3].
\end{cases}
  \end{equation}
 Then the standard Schauder estimate implies the estimate \eqref{est10}.

  We turn to show the uniqueness of the $H^1 (\mb{D})$ weak solution to \eqref{pf13}, that is suppose $G_4 \equiv 0$, $m_i \equiv 0$, $i = 1, 2$, and $\phi \in H^1 (\mb{D})$ is a weak solution to \eqref{pf13}, then $\phi \equiv 0$ in $\mb{D}$.

  Let $\{ \beta_i (y_2) \}_{i = 1}^{\infty}$ be the family of all eigenfunctions to the eigenvalue problem
\begin{equation*}
\begin{cases}
  - \beta_i '' (y_2) = \tau_i^2 \beta_i (y_2), & y_2 \in (-1, 1) ,\\
  \beta_i ' (-1) = \beta_i' (1) = 0. &
\end{cases}
\end{equation*}
 Then
\begin{equation*}
\{ \beta_i (y_2) \}_{i = 0}^{\infty} = \b\{ \f 1{\sqrt{2}} \b\} \cup \b\{ \sqrt{2} \cos (n \pi y_2) \b \}_{n = 0}^{\infty} \cup \b\{ \sin \b( \f {2k +1}2 \pi y_2 \b) \b\}_{n = 0}^{\oo} ,
\end{equation*}
which is a complete orthonormal basis in $L^2 ((-1,1))$ and an orthogonal basis in $H^1 ((-1,1))$. Similarly, the set $\{ \beta_i (y_2) \beta_j (y_3) \}_{i,j = 0}^{\oo}$ will form a complete orthonormal basis in $L^2 ((-1, 1) \times (-1,1))$ and an orthogonal basis in $H^1((-1, 1) \times (-1,1)) $.

Since $\phi \in C^{3, \al} (\overline{\mb{D}})$, then its Fourier series converges
\begin{equation*}
\phi (y_1, y') = \sum_{i, j = 0}^{\oo} X_{i,j} (y_1) \beta_i (y_2) \beta_j (y_3).
\end{equation*}
Substituting this into \eqref{pf13} yields that for $i, j \ge 0$, it holds that
\begin{equation*}
\begin{cases}
 d_1 (y_1) X_{i,j}'' (y_1) + d_2 (y_1) X_{i,j}' (y_1)  - ( \tau_i^2 + \tau_j^2 ) X_{i,j} (y_1) + a_0 b_1 d_5 (y_1) X_{i,j} (L_s)  = 0, \\
 X_{i,j}' (L_s) - b_4 X_{i,j} (L_s) = 0, \\
 X_{i,j}' (L_1) = 0.
\end{cases}
\end{equation*}
Suppose that $X_{i,j} (L_s) = 0$, then $X_{i, j} (y_1) \equiv 0 $ for $\forall y_1 \in [L_s, L_1]$, by the maximum principle and Hopf's lemma. Suppose that $X_{i,j} (L_s) > 0$. Then
\begin{equation}\label{pf15}
\begin{cases}
 d_1 (y_1) X_{i,j}'' (y_1) + d_2 (y_1) X_{i,j}' (y_1)  - ( \tau_i^2 + \tau_j^2 ) X_{i,j} (y_1) = - a_0 b_1 d_5 (y_1) X_{i,j} (L_s)  > 0, \ \ \forall y_1 \in [L_s, L_1], \\
 X_{i,j}' (L_s) = b_4 X_{i,j} (L_s) > 0, \\
 X_{i,j}' (L_1) = 0.
\end{cases}
\end{equation}
Assume that there exists $\tilde{y}_1 \in [L_s, L_1]$, such that $X_{i,j} (\tilde{y}_1) = \max_{y_1 \in [L_s, L_1]} X_{i,j} (y_1) > 0 $. Then the second and the third equations in \eqref{pf15} imply that $\tilde{y}_1 \in ( L_s, L_1]$. If $\tilde{y}_1 \in (L_s, L_1)$, then $X_{i,j}' (\tilde{y}_1) = 0$, $X_{i,j}'' (\tilde{y}_1) \le 0$, which contradicts to the first equation in \eqref{pf15}. If $\tilde{y}_1 = L_1$, then Hopf's lemma yields that $X_{i,j}' (L_1) > 0$, which also contradicts. Similarly, $X_{i,j} (L_s) < 0$ will induce a contradiction. Hence, $X_{i,j} (y_1) \equiv 0$ for all $y_1 \in [L_s, L_1]$. We consequently get $\phi \equiv 0$ in $\mb{D}$. The uniqueness of the $H^1$ weak solution to \eqref{pf13} has been proved.

Using Lemma \ref{lem2} and the Fredholm alternatives for elliptic equations, as well as the argument in \cite[Theorem 8.6]{GT82}, we can deduce that there exists a unique $H^1 (\mb{D})$ weak solution to \eqref{pf13}. Then, the uniqueness helps us to derive the estimate \eqref{est9} from \eqref{est10}. This completes the proof.
\end{proof}

Thus $N_1 (y) = \p_{y_1} \phi(y) + \f {d_4 (y_1)}{b_3} \p_{y_1} \phi (L_s, y') $, $N_2 (y) = \p_{y_2} \phi$ and $N_3 (y) = \p_{y_3} \phi $ is the solution to \eqref{pf9} with \eqref{bc39}-\eqref{bc40}. Differentiating the first equation in \eqref{pf9} with respect to $y_2$ (resp. $y_3$) and evaluating at $y_2 = \pm 1$ (resp. $y_3 = \pm 1$), one gets from \eqref{bc38}, \eqref{bc41}-\eqref{bc42} that
\begin{equation}\no
\begin{cases}
\p_{y_2}^2 N_2 (y_1, \pm 1, y_3) = 0, & \mbox{on } \Sigma_2^{\pm}, \\
   \p_{y_3}^2 N_3 (y_1, y_2, \pm 1) = 0, & \mbox{on } \Sigma_3^{\pm}.
\end{cases}
\end{equation}

Then
\begin{equation}\label{pf32}
 v_1 (y) = \tilde{v}_1 (y) + \p_{y_1} \phi(y) + \f {d_4 (y_1)}{b_3} \p_{y_1} \phi (L_s, y'), \
 v_2 (y) = \tilde{v}_2 (y) + \p_{y_2} \phi (y), \
v_3 (y) = \tilde{v}_3 (y) + \p_{y_3} \phi (y),
\end{equation}
would be the solution to \eqref{pf6} with \eqref{bc34}-\eqref{bc35} and satisfy the estimate
\begin{equation}\label{est12}\begin{aligned}
 \sum_{i = 1}^3 \| v_i \|_{C^{2, \al } (\overline{\mb{D}})}
& \le C_* \b(\sum_{i = 1}^3 \| \tilde{v}_i \|_{C^{2, \al } (\overline{\mb{D}})} + \| \na \phi \|_{C^{2, \al } (\overline{\mb{D}})}  + \| \p_{y_1} \phi (L_s, y') \|_{C^{2, \al} (\overline{E})} \b) \\
& \le
C_* (\eps + \eps \| ({\bf \hat{v}}, \hat{v}_5) \|_{\mc{V}} + \| ({\bf \hat{v}}, \hat{v}_5) \|_{\mc{V}}^2)
\le C_* (\eps + \eps \de_0 + \de_0^2).
\end{aligned}\end{equation}
Moreover, there holds
\begin{equation}\label{bc45}
\begin{cases}
  (v_2, \p_{y_2}^2 v_2, \p_{y_2} v_1, \p_{y_2} v_3) (y_1, \pm 1, y_3) = 0, & \mbox{on } \Sigma_2^{\pm}, \\
  (v_3, \p_{y_3}^2 v_3, \p_{y_3} v_1, \p_{y_3} v_2) (y_1, y_2, \pm 1) = 0, & \mbox{on } \Sigma_3^{\pm}.
\end{cases}
\end{equation}

\textbf{Step 5.} According to \eqref{pf24},
\begin{equation}\label{pf33}
v_4 (y_1, y') = \f {b_2}{b_1} v_1 (L_s, y') + R_4 ({\bf \hat{v}} (L_s, \beta_2 (y), \beta_3 (y)), \hat{v}_5 (\beta_2 (y), \beta_3 (y))),
\end{equation}
we can infer that the function $v_4$ can be uniquely determined by \eqref{pf32}. Furthermore, $v_4$ satisfies the following estimate
\begin{equation}\label{est13}
\| v_4 \|_{C^{2, \al} (\overline{\mb{D}})} \le C_* \| v_1 (L_s, \cdot ) \|_{C^{2, \al } (\overline{E})} + C_* (\eps \| ({\bf \hat{v}}, \hat{v}_5) \|_{\mc{V}} + \| ({\bf \hat{v}}, \hat{v}_5) \|_{\mc{V}}^2 ) \le C_* (\eps \de_0 + \de_0^2),
\end{equation}
and the compatibility conditions
\be\label{bc46}
\begin{cases}
  \p_{y_2} v_4 (y_1, \pm 1, y_3) = \f {b_2}{b_1} \p_{y_2} v_1 (L_s, \pm 1, y_3) = 0, & \mbox{on } (y_1, y_3) \in [L_s, L_1] \times [-1,1], \\
  \p_{y_3} v_4 (y_1, y_2, \pm 1) = \f {b_2}{b_1} \p_{y_3} v_1 (L_s, y_2, \pm 1) = 0, & \mbox{on } (y_1, y_2) \in [L_s, L_1] \times [-1,1].
\end{cases}
\ee

Moreover, according to \eqref{pf1}, the shock front is formulated as
\begin{equation}\label{pf34}
v_5 (y') = \f 1{b_1} v_1 (L_s, y') - \f 1{b_1} R_1 (\hat{{\bf v}} (L_s, y'), \hat{v}_5),
\end{equation}
and $v_5 \in C^{2, \al} (\overline{E})$ satisfies
\begin{equation}\label{bc47}
\begin{cases}
\p_{y_2} v_5 (\pm 1, y_3) = 0, & \mbox{on } y_3 \in [-1,1], \\
  \p_{y_3} v_5 (y_2, \pm 1) = 0, & \mbox{on } y_2 \in [-1,1].
\end{cases}
\end{equation}

Considering that the $C^{2, \al} (\overline{\mb{D}})$ estimate of $R_4$ requires the $C^{3, \al} (\overline{E})$ estimate of $v_5$, we have to improve the regularity of $v_5$ to be $C^{3, \al} (\overline{E})$. To this end, we define
\begin{equation*}
\begin{cases}
  F_2 (y') = \p_{y_2} v_1 (L_s, y') - a_0 b_1 v_2 (L_s, y') - \p_{y_2} R_1 (\hat{{\bf v}} (L_s, y'), \hat{v}_5) - b_1 g_2 (\hat{{\bf v}} (L_s, y'), \hat{v}_5), \\
  F_3 (y') = \p_{y_3} v_1 (L_s, y') - a_0 b_1 v_3 (L_s, y') - \p_{y_3} R_1 (\hat{{\bf v}} (L_s, y'), \hat{v}_5) - b_1 g_3 (\hat{{\bf v}} (L_s, y'), \hat{v}_5).
\end{cases}
\end{equation*}
Then there holds
\begin{equation}\no
 \begin{cases}
          \p_{y_2} F_3 - \p_{y_3} F_2 = 0, & \text{in } E,\\
          \p_{y_2} F_2 + \p_{y_3} F_3 = 0, & \text{in } E, \\
          F_2 (\pm 1 , y_3) = 0, & \text{on } y_3 \in [-1,1], \\
          F_3 (y_2, \pm 1) = 0, & \text{on } y_2 \in [-1,1],
        \end{cases}
\end{equation}
which follows from the first boundary condition in \eqref{pf6} and the boundary conditions \eqref{bc34}-\eqref{bc35}. Using Lemma \ref{lm1}, $F_2 = F_3 \equiv 0$ in $E$. Then \eqref{pf34} implies
\begin{equation}\label{pf36}
\begin{cases}
  \p_{y_2} v_5 (y') = a_0 v_2 (L_s, y') + g_2 (\hat{{\bf v}} (L_s, y'), \hat{v}_5), & \mbox{in } E, \\
  \p_{y_3} v_5 (y') = a_0 v_3 (L_s, y') + g_3 (\hat{{\bf v}} (L_s, y'), \hat{v}_5), & \mbox{in } E.
\end{cases}
\end{equation}
Hence, $v_5 \in C^{3, \al} (\overline{E})$ with the following estimate
\begin{equation}\label{est14}
\begin{aligned}
 \| v_5 \|_{C^{3, \al} (\overline{E})}
&\le C_* (\| v_1 (L_s, \cdot) \|_{C^{2, \al } (\overline{E})} + \| R_1 (\hat{{\bf v}} (L_s, y'), \hat{v}_5) \|_{C^{2, \al } (\overline{E})}) \\
& \q
+ C_* \sum_{i = 2}^3 ( \| v_i (L_s, \cdot) \|_{C^{2, \al } (\overline{E})} + \| g_j (\hat{{\bf v}} (L_s, y'), \hat{v}_5) \|_{C^{2, \al } (\overline{E})} ) \\
& \le
C_* (\eps + \eps \| ({\bf \hat{v}}, \hat{v}_5) \|_{\mc{V}} + \| ({\bf \hat{v}}, \hat{v}_5) \|_{\mc{V}}^2)
\le C_* (\eps + \eps \de_0 + \de_0^2).
\end{aligned}
\end{equation}

Differentiating the first (second) equation in \eqref{pf36} with respect to $y_2$ (resp. $y_3$) twice and evaluating at $y_2 = \pm 1$ (resp. $y_3 = \pm 1$), using \eqref{bc25} and \eqref{bc45}, one gets
\begin{equation}\label{bc48}
\begin{cases}
\p_{y_2}^3 v_5 (y_1, \pm 1, y_3) = 0, & \forall y_3 \in [-1 ,1], \\
  \p_{y_3}^3 v_5 (y_1, y_2, \pm 1) = 0, & \forall y_2 \in [-1 ,1].
\end{cases}
\end{equation}

Combining the estimates \eqref{est12}, \eqref{est13} and \eqref{est14}, one gets
\begin{equation}\no
\| ({\bf v}, v_5) \|_{\mc{V}} = \sum_{i = 1}^4 \| v_i \|_{C^{2, \al} (\overline{\mb{D}})} + \| v_5 \|_{C^{3, \al} (\overline{E})} \le C_* (\eps + \eps \de_0 + \de_)^2) \le C_* (\eps + \de_0^2).
\end{equation}
Choosing $\de_0 = \sqrt{\eps}$ and letting $\eps < \eps_0 = \f 1{4 C_*^2}$, we derive that $\| ({\bf v}, v_5) \|_{\mc{V}} \le 2 C_* \eps \le \de_0$. On the other hand, we have shown that the compatibility conditions \eqref{bc45}, \eqref{bc46}, \eqref{bc47}, \eqref{bc48}, thus $({\bf v}, v_5) \in \mc{V}$. We now can define the operator $\mc{T} ({\bf \hat{v}}, \hat{v}_5) = ({\bf v}, v_5)$, which maps $\mc{V}$ to itself.

\textbf{Step 6. } It remains to verify that $\mc{T}$ is a contraction mapping in the norm
\begin{equation*}
\| ({\bf v}, v_5) \|_{\mc{W}} \co \sum_{i = 1}^4 \| v_i \|_{C^{1, \al} (\overline{\mb{D}})} + \| v_5 \|_{C^{2, \al } (\overline{E})}.
\end{equation*}
 After that the unique fixed point of $\mc{T}$ in $\mc{V}$ is the desired solution. For any two elements $(\hat{{\bf v}}^i, \hat{v}_5^i)$, $i = 1, 2$ in $\mc{V}$, let $({\bf v}^i, v_5^i) = \mc{T} (\hat{{\bf v}}^i, \hat{v}_5^i) $, $i = 1, 2$. Set
\begin{equation*}
(\hat{{\bf z}}, \hat{z}_5) = ({\bf \hat{v}}^1, \hat{v}_5^1) - ({\bf \hat{v}}^2, \hat{v}_5^2), \ \
({\bf z}, z_5) = ({\bf v}^1, v_5^1) - ({\bf v}^2, v_5^2),
\end{equation*}
and $J_i = \tilde{\om}_i^1 - \tilde{\om}_i^2$, $i = 1, 2, 3$ with the vorticity $ (\tilde{\om}_1^j, \tilde{\om}_2^j, \tilde{\om}_3^j ) = \curl {\bf \hat{v}}^j$ for $j = 1, 2$.

It follows from  \eqref{pf2}  that $z_4$ satisfies
\begin{equation}\no
\begin{cases}
  \b( D_1^{\hat{v}_5^1} + \f {\hat{v}_2^1}{\bar{u} (D_0^{\hat{v}_5^1}) + \hat{v}_1^1} D_2^{\hat{v}_5^1} + \f {\hat{v}_3^1}{\bar{u} (D_0^{\hat{v}_5^1}) + \hat{v}_1^1} D_3^{\hat{v}_5^1} \b) z_4 = F_4 (y),  \\
  z_4 (L_s, y') = h_4 (y'),
\end{cases}
\end{equation}
where
\begin{equation*}\begin{aligned}
&
F_4 (y') = - (D_1^{\hat{v}_5^1} - D_1^{\hat{v}_5^2}) v_4^2
- \sum_{i = 2}^3 \b( \f {\hat{v}_i^1}{\bar{u} (D_0^{\hat{v}_5^1}) + \hat{v}_1^1} - \f {\hat{v}_i^2}{\bar{u} (D_0^{\hat{v}_5^2}) + \hat{v}_1^2}
\b) v_4^2, \\
&
h_4 (y') = b_2 z_5 (y') + R_2 ({\bf \hat{v}}^1, \hat{v}_5^1) - R_2 ({\bf \hat{v}}^2, \hat{v}_5^2).
\end{aligned}\end{equation*}
Set $I_2^i$ and $I_3^i$ for $i = 1, 2$ as we defined in \eqref{I2} and \eqref{I3}, where we use $({\bf \hat{v}}^i, \hat{v}_5^i)$ replace $({\bf \hat{v}}, \hat{v}_5)$. Let $(\tau, \bar{y}^i_2 (L_s;y), \bar{y}^i_3 (L_s; y))$ be the trajectory associated with $(1, I_2^i, I_3^i)$, $i = 1, 2$ respectively. Then, together with \eqref{pf33} yield that
\begin{equation}\no
\begin{aligned}
z_4 (y) &= \f {b_2}{b_1} z_1 (L_s, y') + R_4 ({\bf \hat{v}}^1 (L_s, \beta_2^1 (y), \beta_3^1 (y)), \hat{v}_5^1 (\beta_2^1 (y), \beta_3^1 (y))) \\
&\quad- R_4 ({\bf \hat{v}}^2 (L_s, \beta_2^2 (y), \beta_3^2 (y)), \hat{v}_5^2 (\beta_2^2 (y), \beta_3^2 (y))).
\end{aligned}
\end{equation}

Then we estimate the first component of vorticity. It follows from \eqref{pf5} that
\begin{equation}\no
\begin{cases}
  \b( D_1^{\hat{v}_5^1} + \f {\hat{v}_2^1}{\bar{u} (D_0^{\hat{v}_5^1}) + \hat{v}_1^1} D_2^{\hat{v}_5^1} + \f {\hat{v}_3^1}{\bar{u}(D_0^{\hat{v}_5^1}) + \hat{v}_1^1} D_3^{\hat{v}_5^1} \b) J_1  + \mu ({\bf \hat{v}}^1 , \hat{v}_5^1) J_1
 = F_6 (y), \\
  J_1 (L_s, y') = h_6 (y')  ,
\end{cases}
\end{equation}
where
\begin{equation*}
\begin{aligned}
 F_6 (y)& =
- (D_1^{\hat{v}_5^1} - D_1^{\hat{v}_5^2}) \tilde{\om}_1^2
- \sum_{k = 2}^3 \b( \f {\hat{v}_2^1}{\bar{u} (D_0^{\hat{v}_5^1}) + \hat{v}_1^1} D_2^{\hat{v}_5^1} - \f {\hat{v}_2^2}{\bar{u} (D_0^{\hat{v}_5^2}) + \hat{v}_1^2} D_2^{\hat{v}_5^2} \b) \tilde{\om}_1^2 \\
& \q
- ( \mu ({\bf \hat{v}}^1 , \hat{v}_5^1) - \mu ({\bf \hat{v}}^2 , \hat{v}_5^2)) \tilde{\om}_1^2
+ H_0 ({\bf \hat{v}}^1, \hat{v}_5^1) - H_0 ({\bf \hat{v}}^2, \hat{v}_5^2), \\
h_6 (y') &= R_6 ({\bf \hat{v}}^1 (L_s, y'), \hat{v}_5^1 (y')) - R_6 ({\bf \hat{v}}^2 (L_s, y'), \hat{v}_5^2 (y')).
\end{aligned}
\end{equation*}
Thus we obtain the estimate
\begin{equation}\no
\begin{aligned}
\| J_1 \|_{C^{\al} (\overline{\mb{D}})} &\le C_* (\| J_1 (L_s, \cdot) \|_{C^{\al} (\overline{E})} + \| F_6 \|_{C^{\al} (\overline{\mb{D}})} )
\le C_* (\eps + \| \tilde{\om}^2_1 \|_{C^{1, \al} (\overline{\mb{D}})} + \| ({\bf \hat{v}}^2, \hat{v}_5^2) \|_{\mc{V}}) \| ({\bf \hat{z}}, \hat{z}_5) \|_{\mc{W}} \\
&
\le C_* (\eps + \de_0) \| ({\bf \hat{z}}, \hat{z}_5) \|_{\mc{W}}.
\end{aligned}
\end{equation}

Next, we turn to estimate $z_i$ for $i = 1, 2, 3$. The \eqref{pf6} implies that
\begin{equation}\no
\begin{cases}
d_1 (y_1) \p_{y_1} z_1 + \p_{y_2} z_2 + \p_{y_3} z_3 + d_2 (y_1) z_1 + \f {b_2}{b_1} d_3 (y_1) z_1 (L_s, y') =  F_0 (y), & \text{in } \mb{D}, \\
  \p_{y_2} z_3 - \p_{y_3} z_2 + \p_{y_1} \Xi = F_1 (y), & \text{in } \mb{D}, \\
  \p_{y_3} z_1 - \p_{y_1} z_3 - d_4 (y_1) \p_{y_3} z_1 (L_s, y') + \p_{y_2} \Xi = F_2 , & \text{in } \mb{D}, \\
  \p_{y_1} z_2 - \p_{y_2} z_1 + d_4 (y_1) \p_{y_2} z_1 (L_s, y') + \p_{y_3} \Xi = F_3 , & \text{in } \mb{D}, \\
  (\p_{y_2}^2 + \p_{y_3}^2) z_1 (L_s, y') = a_0 b_1 \p_{y_2} z_2 (L_s, y') + a_0 b_1 \p_{y_3} z_3 (L_s, y') + h_1 (y'), & \forall y' \in E, \\
  z_2 (y_1, \pm 1 , y_3) = \Xi (y_1, \pm 1, y_3) = 0, & \text{on } \Sigma_2^{\pm}, \\
  z_3 (y_1, y_2, \pm 1) = \Xi (y_1, y_2, \pm 1) = 0, & \text{on } \Sigma_3^{\pm}, \\
  z_1 (L_1, y') - d_4 (L_1) z_1 (L_s, y') = h_2 (y'), & \forall y' \in E, \\
  \Xi (L_s, y') = \Xi (L_1, y') = 0, & \forall y' \in E,
  \end{cases}
\end{equation}
where
\begin{equation*}\begin{aligned}
& \Xi = \Pi^1 - \Pi^2, \\
& F_0 (y) = G_0 ({\bf \hat{v}}^1, \hat{v}_5^1) - G_0 ({\bf \hat{v}}^2, \hat{v}_5^2), \ \ F_1 (y) = G_1 ({\bf \hat{v}}^1, \hat{v}_5^1) - G_1 ({\bf \hat{v}}^2, \hat{v}_5^2), \\
& F_2 (y) = G_2 ({\bf \hat{v}}^1, \hat{v}_5^1) - G_2 ({\bf \hat{v}}^2, \hat{v}_5^2), \ \  F_3 (y) = G_3 ({\bf \hat{v}}^1, \hat{v}_5^1) - G_3 ({\bf \hat{v}}^2, \hat{v}_5^2), \\
& h_1 (y') = q_1 ({\bf \hat{v}}^1 (L_s, y'), \hat{v}_5^1 (y')) - q_1 ({\bf \hat{v}}^1 (L_s, y'), \hat{v}_5^2 (y')), \\
& h_2 (y') = q_4 ({\bf \hat{v}}^1 (L_s, y'), \hat{v}_5^1 (y')) - q_4 ({\bf \hat{v}}^1 (L_s, y'), \hat{v}_5^2 (y')).
\end{aligned}\end{equation*}
Moreover,
\be\no
 (\p_{y_2} z_1 - a_0 b_1 z_2) (L_s, \pm 1, y_3) = 0 , \q \forall y_3 \in [-1, 1], \\\no
  (\p_{y_3} z_1 - a_0 b_1 z_3) (L_s, y_2, \pm 1) = 0 , \q \forall y_2 \in [-1, 1].
\ee
Similar to the analysis in \textbf{Step 4}, there holds
\begin{equation}\no
\begin{aligned}
\sum_{i = 1}^3 \| z_i \|_{C^{1, \al } (\overline{\mb{D}})}
&\le C_* (\sum_{i = 0}^3 \| F_i \|_{C^{\al } (\overline{\mb{D}})} + \sum_{i = 1}^2 \| h_i \|_{C^{1, \al} (\overline{E})})
\le C_* (\eps + \sum_{i = 1}^2 \| ({\bf \hat{v}}^i, \hat{v}_5^i) \|_{\mc {V}}) \| ({\bf \hat{z}}, \hat{z}_5) \|_{\mc{W}}\\
& \le C_* (\eps + \de_0) \| ({\bf \hat{z}}, \hat{z}_5) \|_{\mc{W}}.
\end{aligned}
\end{equation}

We continue to estimate $z_4$. According to the definition of $R_4$, the following estimate is required.
\begin{equation}\no
\begin{aligned}
& \| (\hat{v}_5^1 (\beta_2^1(y), \beta_3^1 (y)) - \hat{v}_5^1 (y')) - (\hat{v}_5^2 (\beta_2^2(y), \beta_3^2 (y)) - \hat{v}_5^2 (y')) \|_{C^{1, \al} (\overline{\mb{D}})} \\
& \le
\| \hat{z}_5 (\beta_2^1(y), \beta_3^1 (y)) - \hat{z}_5 (y') \|_{C^{1, \al} (\overline{\mb{D}})}
+ \| \hat{v}_5^2 (\beta_2^1(y), \beta_3^1 (y)) -  \hat{v}_5^2 (\beta_2^2(y), \beta_3^2 (y)) \|_{C^{1, \al} (\overline{\mb{D}})} \\
& \le
C_* \sum_{i = 2}^3 \| \beta_i^1 - y_i \|_{C^{1, \al } (\overline{\mb{D}})} \| \hat{z}_5 \|_{C^{2, \al } (\overline{E})}
+ \sum_{i = 2}^3 \| \beta_i^1 - \beta_i^2 \|_{C^{1, \al } (\overline{\mb{D}})} \| \hat{v}_5^2 \|_{C^{2, \al } (\overline{E})}.
\end{aligned}\end{equation}

Denote $Y_i (\tau;y) = \bar{y}_i^1 (\tau; y) - \bar{y}_i^2 (\tau; y)$ for $i = 2, 3$, then $Y_i (L_s; y) = \beta_i^1 (y) - \beta_i^2 (y)$.  It follows from \eqref{ode1} 
and a simple calculation that
\begin{equation}\no
\begin{cases}
  Y_2 (t; y) = \int_{y_1}^t a_{22} (\tau; y) Y_2 (\tau; y) + a_{23} (\tau; y) Y_3 (\tau; y) d \tau + \int_{y_1}^t a_2 (\tau; y) d\tau, \\
  Y_3 (t; y) = \int_{y_1}^t a_{32} (\tau; y) Y_2 (\tau; y) + a_{33} (\tau; y) Y_3 (\tau; y) d\tau + \int_{y_1}^t a_3 (\tau; y) d\tau,
\end{cases}
\end{equation}
where  $a_{ij}$, $i, j = 2, 3$ are functions of ${\bf \hat{v}}^1$, $\hat{v}_5^1$ and $a_i$, $i = 2, 3$ are functions of ${\bf \hat{z}}$, $\hat{z}_5$.
Then the Gronwall's inequality gives that
\begin{equation*}
\sum_{i = 2}^3 \| \beta_i^1 - \beta_i^2 \|_{C^0 (\overline{\mb{D}})} \le C_* (\sum_{j = 1}^4 \| \hat{z}_j \|_{C^0 (\overline{\mb{D}})} + \| \hat{z}_5 \|_{C^1 (\overline{E})}).
\end{equation*}
Similarly, one can derive
\begin{equation*}
\sum_{i = 2}^3 \| \beta_i^1 - \beta_i^2 \|_{C^{1, \al} (\overline{\mb{D}})} \le C_*  \| ({\bf \hat{z}}, \hat{z}_5) \|_{\mc{W}}.
\end{equation*}
Therefore, $z_4$ satisfies the following estimate
\begin{equation}\no
\| z_4 \|_{C^{1, \al} (\overline{\mb{D}})} \le C_* (\| z_1 (L_s, \cdot ) \|_{C^{1, \al } (\overline{E})} + \| R_4^1 - R_4^2 \|_{C^{1, \al } (\overline{\mb{D}})}) \le C_* (\eps + \de_0 )  \| ({\bf \hat{z}}, \hat{z}_5) \|_{\mc{W}}.
\end{equation}

Finally, \eqref{pf34} gives
\begin{equation*}
z_5 (y') = \f 1 {b_1} z_1 (L_s, y') - \f 1 {b_1} (R_1 ({\bf \hat{v}}^1 (L_s, y'), \hat{v}_5^1 (y')) - R_1 ({\bf \hat{v}}^2 (L_s, y'), \hat{v}_5^2 (y'))),
\end{equation*}
from which we can infer that
\begin{equation}\no
\begin{aligned}
\| z_6 \|_{C^{1, \al } (\overline{E})} &\le C_* (\| z_1 (L_s, \cdot) \|_{C^{1, \al } (\overline{E})} + \| R_1 ({\bf \hat{v}}^1 (L_s, y'), \hat{v}_5^1 (y')) - R_1 ({\bf \hat{v}}^2 (L_s, y'), \hat{v}_5^2 (y')) \|_{C^{1, \al } (\overline{\mb{D}})})\\
& \le C_*  (\eps + \de_0 )  \| ({\bf \hat{z}}, \hat{z}_5) \|_{\mc{W}}.
\end{aligned}
\end{equation}

It follows from \eqref{pf36} that
\begin{equation*}
\begin{cases}
  \p_{y_2} z_5 (y') = a_0 z_2 (L_s, y') + g_2 ({\bf \hat{v}}^1 (L_s, y'), \hat{v}_5^1 (y')) - g_2 ({\bf \hat{v}}^2 (L_s, y'), \hat{v}_5^2 (y')) , & \mbox{in } E, \\
  \p_{y_3} z_5 (y') = a_0 z_3 (L_s, y') + g_3 ({\bf \hat{v}}^1 (L_s, y'), \hat{v}_5^1 (y')) - g_3 ({\bf \hat{v}}^2 (L_s, y'), \hat{v}_5^2 (y')) , & \mbox{in } E.
\end{cases}
\end{equation*}
Then
\begin{equation}\no
\begin{aligned}
&\| (\p_{y_2} z_5, \p_{y_3} z_5) \|_{C^{1, \al } (\overline{E})} \\
&\le C_* \sum_{i = 2}^3 (\| z_i (L_s, y') \|_{C^{1, \al } (\overline{E})} + \| g_i ({\bf \hat{v}}^1 (L_s, y'), \hat{v}_5^1 (y')) - g_i ({\bf \hat{v}}^2 (L_s, y'), \hat{v}_5^2 (y')) \|_{C^{1, \al } (\overline{E})} ) \\
& \le C_* (\eps + \de_0 )  \| ({\bf \hat{z}}, \hat{z}_5) \|_{\mc{W}}.
\end{aligned}
\end{equation}
Combining all the above estimates leads to
\begin{equation}\no
\| ({\bf z}, z_5) \|_{\mc{W}} \le C_* (\eps + \de_0 )  \| ({\bf \hat{z}}, \hat{z}_5) \|_{\mc{W}}.
\end{equation}
Since $\de_0 = \sqrt{\eps}$, choosing $\eps < \eps_0 = \f 1{16 C_*}$, then $\| ({\bf z}, z_5) \|_{\mc{W}} \le \f 12 \| ({\bf \hat{z}}, \hat{z}_5) \|_{\mc{W}}$ and $\mc{T}$ is a contraction mapping in the weak norm $\| \cdot \|_{\mc{W}}$. Then there exists a unique fixed point $({\bf v}, v_5) \in \mc{V}$ such that $\mc{T} ({\bf \hat{v}}, \hat{v}_5) = ({\bf v}, v_5)$.

It remains to prove that the auxiliary function $\Pi$ associating with the fixed point $({\bf v}, v_5)$ satisfies $\Pi \equiv 0 $ in $\mb{D}$. Indeed, due to the definitions of $G_i ({\bf v}, v_5)$ for $i = 1, 2, 3$, there is
\begin{equation}\label{pf40}
\begin{cases}
  - \p_{y_1} \Pi = D_2^{v_5} v_3 - D_3^{v_5} v_2 - \tilde{\om}_1,  \\
  - \p_{y_2} \Pi = D_3^{v_5} v_1 - D_1^{v_5} v_3 - \f {v_2 \tilde{\om}_1}{\bar{u} (D_0^{v_5}) + v_1} - \f 1{\bar{u} (D_0^{v_5}) + v_1} D_2^{v_5} v_4, \\
  - \p_{y_3} \Pi = D_1^{v_5} v_2 - D_2^{v_5} v_1 - \f {v_3 \tilde{\om}_1}{\bar{u} (D_0^{v_5}) + v_1} + \f 1{\bar{u} (D_0^{v_5}) + v_1} D_2^{v_5} v_4.
\end{cases}
\end{equation}
Since $\tilde{\om}_1$ satisfies \eqref{om7} and
\begin{equation*}
D_1^{v_5} D_2^{v_5} = D_2^{v_5} D_1^{v_5}, \,
D_2^{v_5} D_3^{v_5} = D_3^{v_5} D_2^{v_5}, \,
D_1^{v_5} D_3^{v_5} = D_3^{v_5} D_1^{v_5},
\end{equation*}
one can infer from \eqref{pf40} that
\begin{equation*}
- D_1^{v_5} (\p_{y_1} \Pi) - D_2^{v_5} (\p_{y_2} \Pi) - D_3^{v_5} (\p_{y_3} \Pi) = 0, \q \text{in } \mb{D}.
\end{equation*}
Since $\| v_5 \|_{C^{3, \al } (\overline{D})} \le \de_0$ with sufficiently small $\de_0$, then $\Pi$ satisfies a second order uniformly elliptic equation without zeroth order term. It follows from $\Pi = 0$ on $\p \mb{D}$ and the maximum principle that $\Pi \equiv 0$ in $\mb{D}$. Thus $({\bf v}, v_5)$ is the solution which we desire. We complete the proof of Theorem \ref{main}.

\section{Appendix}\label{appendix}\noindent
\par In this appendix, we give the explicit expressions of $ J ({\bf v} (L_s, y'), v_5 (y'))$, $ J_i ({\bf v} (L_s, y'), v_5 (y'))(i=2,3) $ and $R_{0i}$, $i = 1, 2, 3$ needed in \eqref{rh12}-\eqref{euler6}.
\begin{equation*}
\begin{aligned}
& J ({\bf v} (L_s, y'), v_5 (y')) = \b(\tilde{\rho} ({\bf v} (L_s, y'), v_5 (y')) v_2^2 (L_s, y') + \tilde{P} ({\bf v} (L_s, y'), v_5 (y')) - (\rho^- (u_2^-)^2 + P^-) (L_s + v_5, y')\b) \\
&\q
\times
\b(\tilde{\rho} ({\bf v} (L_s, y'), v_5 (y')) v_3^2 (L_s, y') + \tilde{P} ({\bf v} (L_s, y'), v_5 (y')) - (\rho^- (u_3^-)^2 + P^-) (L_s + v_5, y')\b) \\
& \q
-
\b[ \tilde{\rho}({\bf v} (L_s, y'), v_5 (y')) ( v_2 v_3) (L_s, y')
- (\rho^-  u_2^- u_3^-) (L_s + v_5, y')
\b]^2
, \end{aligned}\end{equation*}
\begin{equation*}
\begin{aligned}
& J_2 ({\bf v} (L_s, y'), v_5 (y')) =
\b( \tilde{\rho}  ({\bf v} (L_s, y'), v_5 (y')) v_3^2 (L_s, y') + \tilde{P} ({\bf v} (L_s, y'), v_5 (y'))
- (\rho^- (u_3^-)^2 + P^-) (L_s + v_5, y')
\b) \\
& \q
\times
\b( \tilde{\rho} ({\bf v} (L_s, y'), v_5 (y')) (v_1 (L_s, y') + \bar{u} (L_s + v_5)) v_2 (L_s, y')
- (\rho^- u_1^- u_2^-) (L_s, y')
\b) \\
& \q
- \b( \tilde{\rho} ({\bf v} (L_s, y'), v_5 (y')) (v_1 (L_s, y') + \bar{u} (L_s + v_5)) v_3 (L_s, y') - (\rho^- u_1^- u_3^-) (L_s, y')
\b) \\
& \q \times
\b(
\tilde{\rho} ({\bf v} (L_s, y'), v_5 (y')) (v_2 v_3) (L_s, y') - (\rho^- u_2^- u_3^-) (L_s, y')
\b), \\
& J_3 ({\bf v} (L_s, y'), v_5 (y')) =
\b( \tilde{\rho}  ({\bf v} (L_s, y'), v_5 (y')) v_2^2 (L_s, y') + \tilde{P} ({\bf v} (L_s, y'), v_5 (y'))
- (\rho^- (u_2^-)^2 + P^-) (L_s + v_5, y')
\b) \\
& \q
\times
\b( \tilde{\rho} ({\bf v} (L_s, y'), v_5 (y')) (v_1 (L_s, y') + \bar{u} (L_s + v_5)) v_3 (L_s, y')
- (\rho^- u_1^- u_3^-) (L_s, y')
\b) \\
& \q
- \b( \tilde{\rho} ({\bf v} (L_s, y'), v_5 (y')) (v_1 (L_s, y') + \bar{u} (L_s + v_5)) v_2 (L_s, y') - (\rho^- u_1^- u_3^-) (L_s, y')
\b) \\
& \q \times
\b(
\tilde{\rho} ({\bf v} (L_s, y'), v_5 (y')) (v_2 v_3) (L_s, y') - (\rho^- u_2^- u_3^-) (L_s, y')
\b),
\end{aligned}
\end{equation*}
and
\begin{equation*}
\begin{aligned}
& R_{01} ({\bf v} (L_s, y'), v_5 (y')) =
- [\bar{\rho} \bar{u}] (L_s + v_5)
+ (\rho^- u_1^- ) (L_s + v_5, y') - (\bar{\rho}^- \bar{u}^-) (L_s + v_5)\\
&\q
+ \sum_{i = 2}^3 (\rho ({\bf v} (L_s, y'), v_5 (y')) v_i  (L_s, y') - (\rho^- u_i^-) (L_s + v_5, y')) \f {J_i ({\bf v} (L_s, y'), v_5 (y'))}{J ({\bf v} (L_s, y'), v_5 (y'))}\\
&\q
- (v_1 (L_s , y') + \bar{u}^+ (L_s + v_5) - \bar{u}^+ (L_s)) (\tilde{\rho} ({\bf v} (L_s, y'), v_5) - \bar{\rho}^+ (L_s + v_5))
- (\bar{\rho}^+ (L_s + v_5) - \bar{\rho}^+ (L_s)) v_1 (L_s, y'),
\end{aligned}\end{equation*}
\begin{equation*}\begin{aligned}
& R_{02} ({\bf v} (L_s, y'), v_5 (y')) =
- \left([\bar{\rho} \bar{u}^2+\bar P] (L_s+v_5)-((\bar \rho^+-\bar \rho^-)\bar f)(L_s)v_5\right)
 \\
& \q
+ (\rho^- (u_1^-)^2 + P^-) (L_s + v_5, y')
- (\bar{\rho}^- (\bar{u}^-)^2 + \bar{P}^-) (L_s + v_5) \\
& \q
- \b(
\tilde{\rho} ({\bf v} (L_s, y'), v_5 (y')) (v_1 (L_s + v_5, y') + \bar{u}^+ (L_s + v_5))^2 + \tilde{P} ({\bf v} (L_s, y'), v_5 (y'))
- (\bar{\rho} (\bar{u}^+)^2 + \bar{P}^+) (L_s + v_5)\\\no
& \q
- [(\bar{u}^+ (L_s))^2 + c^2 (\bar{\rho}^+ (L_s))] (\tilde{\rho} ({\bf v} (L_s, y'), v_5 (y')) - \bar{\rho}^+ (L_s + v_5))
+ 2 (\bar{\rho}^+ \bar{u}^+ ) (L_s) v_1 (L_s, y')
\b)\\
& \q
+ \sum_{i = 2}^3
(\tilde{\rho} ({\bf v} (L_s, y'), v_5 (y')) (v_1 (L_s, y') + \bar{u} (L_s + v_5)) v_i (L_s, y') - (\rho^- u_1^- u_i^-) (L_s + v_5, y') )
\f {J_i ({\bf v} (L_s, y'), v_5 (y'))}{J ({\bf v} (L_s, y'), v_5 (y'))}, \\
& R_{03} ({\bf v} (L_s, y'), v_5 (y')) =
(\bar{u}^+ (L_s + v_5) - \bar{u}^+ (L_s)) v_1 (L_s , y')
- \f {c^2 (\bar{\rho}^+ (L_s))}{\bar{\rho}^+ (L_s)} (\tilde{\rho} ({\bf v} (L_s, y'), v_5 (y')) - \bar{\rho}^+ (L_s + v_5) )
 \\
& \q
+ \f {\ga }{\ga - 1}
((\tilde{\rho}({\bf v} (L_s, y'), v_5 (y')))^{\ga - 1} - (\bar{\rho}^+ (L_s + v_5))^{\ga - 1})
+ \f 12 \sum_{i = 1}^3 v_i^2 (L_s , y')
- \eps \Phi_0 (L_s, y').
\end{aligned}
\end{equation*}
\par {\bf Acknowledgement.} Weng is partially supported by National Natural Science Foundation of China 12071359, 12221001.

\end{document}